\documentclass[10pt]{amsart}
\usepackage[utf8]{inputenc}

\usepackage{amssymb,amsthm,amsmath}
\usepackage{mathrsfs}
\usepackage{enumerate}
\usepackage{graphicx,xcolor}
\usepackage{setspace}
\usepackage[hidelinks]{hyperref}
\usepackage{comment}

\usepackage{pgfplots}
\usepgfplotslibrary{fillbetween}
\usetikzlibrary{calc}
\pgfplotsset{compat=1.18}

\usepackage[paper=a4paper, left=1.2in, right=1.2in, top=1.2in, bottom=1.2in]{geometry}
\newcommand{\ls}{\leqslant}
\newcommand{\gr}{\geqslant}
\newcommand{\E}{\mathbb{E}}
\newcommand{\pp}{\mathbb{P}}
\newcommand{\1}{\textbf{1}}
\newcommand{\R}{\mathbb{R}}

\newcommand{\N}{\mathbb{N}}
\newcommand{\Sp}{\mathbb{S}}

\newcommand{\e}{\varepsilon}
\newcommand{\n}{\mathbf{n}}

\DeclareMathOperator{\sgn}{sgn}
\DeclareMathOperator{\vol}{vol}

\numberwithin{equation}{section} 

\newtheorem{theorem}{Theorem}[section] 
\newtheorem{lemma}[theorem]{Lemma}
\newtheorem{corollary}[theorem]{Corollary}
\newtheorem{proposition}[theorem]{Proposition}

\theoremstyle{remark}

\theoremstyle{definition}
\newtheorem{defn}[theorem]{Definition}

\hypersetup{
    pdftitle={Stability of Hyperplane Projections of the Regular Simplex and Lp-Projection Bodies of balls in lqn(R)},
    pdfauthor={Jacek Jakimiuk, Christos Pandis},
    pdfkeywords={Regular simplex, hyperplane projections, stability estimates, Lp-projection bodies},
    colorlinks,
   linkcolor={violet!50!black},
citecolor={orange!75!black},
urlcolor={teal!65!black}
}

\title{Stability of Hyperplane Projections of the Regular Simplex and $L_p$-Projection Bodies of balls in $\ell_q^n(\R)$}

\author{Jacek Jakimiuk and Christos Pandis}

\address{Institute of Mathematics, University of Warsaw, 02-097, Warsaw, Poland}

\email{jj406165@mimuw.edu.pl}

\address{Department of Mathematics and Applied Mathematics, University of Crete, Voutes Campus, 70013 Heraklion, Greece}

\email{chrpandis@gmail.com}



\begin{document}

\begin{abstract}
We prove stability estimates for extremal hyperplane projections
of the regular simplex. We further generalize the analogous stability estimates for $\ell_q$-balls by the first named author and the monotonicity theorem for normalized ratios by Barthe and Naor to the
$L_p$-setting. To this end, we derive closed formulas for the
support functions of their $L_p$-projection bodies, complementing
those obtained by the second named author. A key ingredient in our proofs
is a stability refinement of the majorization inequality for
Gaussian mixtures.
\end{abstract}
    
\maketitle

\section{Introduction}

For $1 \le q \le \infty$, let
\begin{align*}
    B_q^n=\{x\in\R^n:\|x\|_q\leq 1\}
\end{align*}
be the unit ball in the standard $\ell_q^n$ norm. The problem of determining central hyperplane sections and projections of extremal volume of certain convex bodies, such as the regular simplex or the unit ball in the $\ell_q^n(\R)$ norm, is a significant problem in convex geometry. Given the extensive literature on this topic, we refer the interested reader to the comprehensive survey \cite{nayar2023extremal} for a detailed historical account of the case of the cube, the current (in 2023) state of the art, and numerous intriguing open problems. For the main advances made after the aforementioned survey was published, we refer to \cite{eskenazis2024resilience, koenig2025hyperplane}. Of particular relevance to the present work is the result of Barthe and Naor \cite{barthe2002hyperplane} on the monotonicity in $q$ of the function
\begin{align}\label{eq:barthe-naor-monotonicity}
    q\longmapsto
\frac{\operatorname{vol}_{n-1}\left(\operatorname{Proj}_{a^\perp}B_q^n\right)}
     {\operatorname{vol}_{n-1}\left(B_q^{n-1}\right)},
\end{align}
where $q \in [1, \infty]$ and $\operatorname{Proj}_{a^{\perp}}K$ is a projection of a convex body $K$ onto a hyperplane $a^{\perp}$ orthogonal to a fixed vector $a$.

Once the extremal values and directions are known, one may ask for quantitative stability estimates. This question for extremal hyperplane sections of $\ell_q$-balls was systematically studied by Chasapis, Nayar and Tkocz \cite{chasapis2022slicing}, who obtained robust versions in all cases in which the extremizers were known. See also \cite{melbourne2022quantitative,eskenazis2024resilience} for related quantitative stability results concerning cube slicing. Stability of the minimal hyperplane projection for the cross-polytope $B_1^n$ follows from the robust form of Szarek's inequality established by De, Diakonikolas and Servedio \cite{de2016robust}. More precisely, there exists a universal constant $\kappa>0$ such that, for every $a\in \Sp^{n-1}$ with $a_1 \ge a_2 \ge \ldots \ge a_n \ge 0$,
\[
\mathbb E\left|
\sum_{j=1}^n a_j\varepsilon_j
\right|
\geq
\frac{1}{\sqrt2}
+
\kappa \|a-a_{(2)}\|_2^2,
\]
where $\varepsilon_1,\ldots,\varepsilon_n$ are independent Rademacher
random variables and $a_{(2)}=(1/\sqrt{2},1/\sqrt{2},0,\ldots,0)$. As discussed in \cite{eskenazis2024resilience}, by the known hyperplane projection volume formula for $B_1^n$, this robust Khintchine inequality is equivalent to stability of the minimal hyperplane projections of $B_1^n$. Recently, in \cite{jakimiuk2026stability}, the first named author established stability estimates for extremal hyperplane projections of the unit balls of $\ell_q^n(\R)$ (except for minimal projections when $1 < q < 2$), providing a dual counterpart to the corresponding results of \cite{chasapis2022slicing} for sections. The stability of sections of the regular simplex was recently developed by Myroshnychenko, Tang, Tatarko and Tkocz \cite{myroshnychenko2026stability}. The stability problem for projections of the simplex has not been solved before the present work.

There are several ways to extend sharp projection estimates and
their stability refinements for convex bodies such as the
$\ell_q$-balls $B_q^n$ and the regular simplex $\Delta_n$.
One direction is to consider projections onto lower-dimensional
subspaces. Sharp estimates for such projections have been
extensively studied; see, for instance, the survey \cite{nayar2023extremal}. Another approach, considered in \cite{pandis2026extremal}, is to replace the
support function $h_{\Pi K}$ of the classical projection body,
which represents the volumes of hyperplane projections, by the
support function $h_{\Pi_p K}$ of the $L_p$-projection body.
This provides an $L_p$ extension of the classical projection
problem, which is recovered when $p=1$. We refer to \cite{schneider2013convex,gardner2006geometric} for more information on the $L_p$ theory.

In this paper we show two stability results. First, we establish the quantitative stability of extremal hyperplane projections of the regular simplex. Second, we generalize some of the known stability results for projections of balls $B_q^n$ to their $L_p$-projection bodies. Moreover, we extend the monotonicity of \eqref{eq:barthe-naor-monotonicity} to $L_p$-projection bodies.

\subsection{Notation}

In this paper we have two different notions of $p$-norm, the one referring to $L_p$-projection bodies and the other referring to a convex body $B_p^n$ being the unit ball in the usual $\ell_p$-norm in $\R^n$. To avoid confusion, throughout the whole paper we use the convention that for the first case we use the letter $p$ and for the second we use the letter $q$ ($q_1$, $q_2$), so that we always write $B_q^n$ ball, $\ell_q$-norm and so on.

By $\langle \cdot, \cdot \rangle$ we denote the usual scalar product in $\R^n$. Let $e_1, \ldots, e_n$ denote the standard orthonormal basis in $\R^n$. We write $a_{(k)} = \frac{1}{\sqrt{k}}\sum_{i=1}^ke_i$. By $\vol_n$ we denote the $n$-dimensional Lebesgue measure. By $\E_X$, $\E_Y$, etc. we denote expectation with respect to the random variables $X$ or $X_i$, $Y$ or $Y_i$ etc., using this notation only if we want to emphasize which random variables are fixed and which are randomized. We use the same convention for $\pp_X$, $\pp_Y$, etc.

\subsection{Stability estimates for the regular simplex}

Let $\Delta_n$ denote the regular $n$-dimensional simplex of side length $\sqrt{2}$, which we view through its usual embedding into $\mathcal{H}=\left \{\sum_{i=1}^{n+1}x_i=1 \right\}$ in $\mathbb{R}^{n+1}$, namely
\[
\Delta_n = \left\{x\in\mathbb{R}^{n+1}: x_j\gr 0, \, \sum_{j=1}^{n+1}x_j=1\right\}.
\]
The problem of extremal hyperplane projections can be described as follows: Find
\[
\max_{H}\,/\,\min_{H}\, \operatorname{vol}_{n-1}\left(\operatorname{Proj}_{H}\Delta_n \right),
\]
where the maximum (respectively, minimum) is taken over all affine subspaces 
$H$ of $\mathcal{H}$ of relative codimension one. Since the volume of a
projection is invariant under translations, we may assume that these
subspaces pass through the barycenter $\text{bar}(\Delta_n)$. Every such
hyperplane extends to a codimension-one subspace of $\mathbb{R}^{n+1}$ by
taking the affine hull of $H$ and the origin, yielding a hyperplane
$a^{\perp}$ with $a$ lying in
$\mathcal{H}'=\left\{\sum_{i=1}^{n+1}x_i=0 \right\}$. The main difference
from the corresponding slicing problem (see \cite{webb1996central}) is that
the barycenter condition is not imposed a priori but arises naturally. Thus,
the problem can be reformulated as follows:
\[
\max_{a\in \mathcal{H}'\cap S^{n}}\,/\,\min_{a\in \mathcal{H}'\cap S^{n}}\, \operatorname{vol}_{n-1}\left(\operatorname{Proj}_{a^{\perp}\cap \mathcal{H}}\Delta_n \right).
\]
The extremal values and directions were first determined by Martini and
Weissbach \cite{weissbach1984besten}; see also
\cite{filliman1990extreme} for the maximal case. Recently, the second-named author presented in
\cite[Theorem~1.2]{pandis2026extremal} an explicit coordinate
formula, valid in every direction, together with an
elementary proof. The formula follows directly from an
identity of Martini and Weissbach \cite{martini1992quermasses}
and recovers the known sharp volume bounds and all equality cases.

Let $a\in S^n\cap\mathcal H'$. Then, 
\begin{equation}\label{eq:simplex-projection-formula}
\operatorname{vol}_{n-1}
\left(
\operatorname{Proj}_{a^\perp\cap\mathcal H}\Delta_n
\right)
=
\frac{\sqrt{n+1}}{2(n-1)!}
\sum_{j=1}^{n+1}|a_j|.
\end{equation}

As a consequence of~\eqref{eq:simplex-projection-formula},
\[
\frac{\sqrt{n+1}}{\sqrt{2}(n-1)!}
\leq
\operatorname{vol}_{n-1}
\left(
\operatorname{Proj}_{a^\perp\cap\mathcal H}\Delta_n
\right)
\leq
\begin{cases}
\displaystyle
\frac{n+1}{2(n-1)!},
& n \text{ odd},\\[2mm]
\displaystyle
\frac{\sqrt{n(n+2)}}{2(n-1)!},
& n \text{ even}.
\end{cases}
\]
Equality in the lower estimate holds, up to permutation and multiplication
by $-1$, at
\[
a_\ast=\frac{e_1-e_2}{\sqrt2}.
\]
If $n$ is odd, equality in the upper estimate holds at
\[
a_{\mathrm{odd}}
=
\frac1{\sqrt{n+1}}
\left(
\underbrace{1,\ldots,1}_{(n+1)/2},
\underbrace{-1,\ldots,-1}_{(n+1)/2}
\right).
\]
If $n$ is even, equality in the upper estimate holds at
\[
a_{\mathrm{even}}
=
\left(
\underbrace{
\sqrt{\frac{n+2}{n(n+1)}},\ldots,
\sqrt{\frac{n+2}{n(n+1)}}}_{n/2},
\underbrace{
-\sqrt{\frac{n}{(n+1)(n+2)}},\ldots,
-\sqrt{\frac{n}{(n+1)(n+2)}}}_{(n+2)/2}
\right),
\]
again up to permutation and multiplication by $-1$.

Once the extremal directions are known, a natural quantitative question
is whether a nearly extremal projection must have its normal close to an
extremizing direction. This is the corresponding stability problem.

For the regular simplex, the exact projection formula
\eqref{eq:simplex-projection-formula}, obtained in
\cite{pandis2026extremal}, reduces the stability problem to a
quantitative analysis of the $\ell_1$-norm on
$S^n\cap\mathcal H'$. We use this formula to prove stability estimates
for both extremal problems.

\begin{theorem}\label{thm:simplex-projection-stability}
Let $n\geq2$ and let $a\in S^n\cap\mathcal H'$. Then
\begin{equation}\label{eq:sharp-lower-stability}
\frac{
\operatorname{vol}_{n-1}
\left(
\operatorname{Proj}_{a^\perp\cap\mathcal H}\Delta_n
\right)
}{
\operatorname{vol}_{n-1}
\left(
\operatorname{Proj}_{a_\ast^\perp\cap\mathcal H}\Delta_n
\right)
}
\geq
1+
\frac18\,
\min_{i\neq j}
\left\|
a-\frac{e_i-e_j}{\sqrt2}
\right\|_2.
\end{equation}
The exponent $1$ is optimal.
For the upper stability estimates, after multiplying $a$ by $-1$, if
necessary, and permuting its coordinates, assume that
\[
a_1\geq a_2\geq\cdots\geq a_{n+1}
\qquad\text{and}\qquad
\#\{k:a_k>0\}
\leq
\left\lfloor\frac{n+1}{2}\right\rfloor.
\]

If $n$ is odd, then
\begin{equation}\label{eq:odd-optimal-stability}
\frac{
\operatorname{vol}_{n-1}
\left(
\operatorname{Proj}_{a^\perp\cap\mathcal H}\Delta_n
\right)
}{
\operatorname{vol}_{n-1}
\left(
\operatorname{Proj}_{a_{\mathrm{odd}}^\perp\cap\mathcal H}\Delta_n
\right)
}
\leq
1-\frac{1}{2(n+2)}
\|a-a_{\mathrm{odd}}\|_2^2.
\end{equation}
If $n$ is even, then
\begin{equation}\label{eq:even-optimal-stability}
\frac{
\operatorname{vol}_{n-1}
\left(
\operatorname{Proj}_{a^\perp\cap\mathcal H}\Delta_n
\right)
}{
\operatorname{vol}_{n-1}
\left(
\operatorname{Proj}_{a_{\mathrm{even}}^\perp\cap\mathcal H}\Delta_n
\right)
}
\leq
1-\frac{1}{n+2}
\|a-a_{\mathrm{even}}\|_2^2.
\end{equation}
The exponent $2$ in both upper stability estimates is optimal.
\end{theorem}

\subsection{Stability estimates for $L_p$-projection bodies}

Let $K$ be a convex body in $\mathbb{R}^n$. The projection body $\Pi K$ of $K$ is a symmetric convex body in $\mathbb{R}^n$ defined through the support function
\[
h_{\Pi K}(a):=\frac{1}{2}\int_{\mathbb{S}^{n-1}}|\langle a,x\rangle|\,dS_{n-1}(K,x).
\]

Let $1\ls p<\infty$ and let $K$ be a convex body in $\mathbb{R}^n$ with $0\in\operatorname{int}(K)$. Lutwak, Yang, and Zhang \cite{lutwak2000lp} introduced the $L_p$-projection body $\Pi_pK$. Under a different normalization, it is defined by
\begin{equation}\label{L_p proj body}
h_{\Pi_p K}^p(a)=\frac{1}{2}\int_{\mathbb{S}^{n-1}}|\langle a,v\rangle|^p\,dS_{p}(K,v),
\end{equation}
for $a\in \mathbb{S}^{n-1}$, where the $L_p$ surface area measure $S_p(K,\cdot)$, the analogue of the surface area measure in the $L_p$ Brunn--Minkowski theory, is defined by
\[
dS_{p}(K,\cdot)=h_{K}^{1-p}\,dS_{n-1}(K,\cdot).
\]

In \cite{pandis2026extremal}, the classical hyperplane projection
formulas for the cube, the cross-polytope, and the regular simplex
(see \cite{nayar2023extremal}) were extended to the setting of
$L_p$-projection bodies. More precisely, for $p\geq 1$ and $a\in\mathbb \Sp^{n-1}$, one has
\begin{equation}\label{eq:lp-projection-formulas}
h_{\Pi_p B_{\infty}^n}^p(a)
=
2^{n-1}\sum_{i=1}^n |a_i|^p
\end{equation}
and
\begin{equation}\label{eq:lp-cross-polytope-formula}
h_{\Pi_p B_1^n}^p(a)
=
\frac{2^{n-1}}{(n-1)!}
\E\left|\sum_{i=1}^n a_i\varepsilon_i\right|^p,
\end{equation}
where $\varepsilon_1,\ldots,\varepsilon_n$ are independent
Rademacher random variables. For the centered regular simplex
$\Delta_c^n=\Delta_n-\operatorname{bar}(\Delta_n)$, one similarly has
\begin{equation}\label{eq:lp-simplex-formula}
h_{\Pi_p\Delta_c^n}^p(a)
=
\frac{(n+1)^{(2p-1)/2}}{2(n-1)!}
\sum_{i=1}^{n+1}|a_i|^p,
\qquad a\in\mathbb S^n\cap\mathcal H'\subset \R^{n+1}.
\end{equation}
Notice that since $\Delta_c^n$ is contained in $\mathcal{H}'$ the support function, surface-area measure, and $L_p$-projection body are taken intrinsically
in the Euclidean space $\mathcal{H}'$.

These formulas are accompanied, in most cases, by sharp estimates
for the corresponding support functions.

We next give a closed probabilistic formula for the support functions of the $L_p$-projection bodies of $\ell_q$-balls, $1<q<\infty$, which shall be used to obtain stability estimates for the corresponding normalized support ratios.

\begin{proposition}\label{prop:Lp-projection-formula}
Let $1 \le p < \infty$, $1 < q < \infty$ and $n \in \N$. Let also $X_1, \ldots, X_n$ be i.i.d. random variables with density proportional to $|t|^{\frac{2-q}{q-1}}e^{-|t|^{\frac{q}{q-1}}}$. Then there exists a constant $C_{p, q, n}$, depending only on $p$, $q$ and $n$, such that for any $a = (a_1, \ldots, a_n) \in \Sp^{n-1}$ we have
\begin{align}\label{eq:Lp-projection-probabilistic-formula}
    h_{\Pi_pB_q^n}^p(a) = C_{p, q, n}\E\left|\sum_{i=1}^na_iX_i\right|^p.
\end{align}
\end{proposition}

We defer the proof of Proposition \ref{prop:Lp-projection-formula} to Section \ref{sec:3}.

The stability estimates of \cite{jakimiuk2026stability} can be stated as stability estimates
for ratios of support functions of the form
\[
\frac{h_{\Pi B_q^n}(a)}
     {h_{\Pi B_q^n}(a_{\mathrm{ext}})},
\]
where $a_{\mathrm{ext}}$ denotes the relevant extremal direction. This suggests a natural generalization, simply by replacing $\Pi B_q^n$ with $\Pi_pB_q^n$. 

\begin{theorem}\label{thm:lp-projection-stability}
Let $n\geq2$, $1\leq p<\infty$, $p\neq2$, $q = 1$ or\, $2 < q \le \infty$, and $a\in \Sp^{n-1}$.  After changing signs and permuting coordinates,
assume that $a_1\geq a_2\geq\cdots\geq a_n\geq0.$

For $q = \infty$, if $1\leq p<2$, let \[
c_{p,\infty}
=
2^{-p/2}\min\left\{
\frac{1-2^{p/2-1}}{2p},
2^{\frac{2-p}{2p}}-1
\right\}.
\]
Then,
\begin{equation}\label{eq:cube-diagonal-stability-below-two}
\frac{h_{\Pi_pB_{\infty}^n}(a)}
     {h_{\Pi_pB_{\infty}^n}\left(a_{(n)}\right)}
\leq
1-\frac{2-p}{4pn}
\sum_{i,j=1}^n
\bigl(a_i-a_j\bigr)^2,
\end{equation}
and
\begin{equation}\label{eq:cube-coordinate-stability-below-two}
\frac{h_{\Pi_pB_{\infty}^n}(a)}
     {h_{\Pi_pB_{\infty}^n}\left(a_{(1)}\right)}
\geq
1+c_{p, \infty}\|a-a_{(1)}\|_2^p.
\end{equation}
If $p>2$, then
\begin{equation}\label{eq:cube-diagonal-stability-above-two}
\frac{h_{\Pi_pB_{\infty}^n}(a)}
     {h_{\Pi_pB_{\infty}^n}\left(a_{(n)}\right)}
\geq
1+\frac{p-2}{4pn}
\sum_{i,j=1}^n
\bigl(a_i-a_j\bigr)^2,
\end{equation}
and
\begin{equation}\label{eq:cube-coordinate-stability-above-two}
\frac{h_{\Pi_pB_{\infty}^n}(a)}
     {h_{\Pi_pB_{\infty}^n}\left(a_{(1)}\right)}
\leq
1-\frac{p-2}{2p}\|a-a_{(1)}\|_2^2.
\end{equation}

For $q = 1$, let
\[
C_0=\frac{7+\sqrt{17}}{2}
\qquad\text{and}\qquad
c_{p, 1}=\frac{|p-2|}{4pC_0}
=\frac{|p-2|}{2p(7+\sqrt{17})}.
\]
If $1\leq p<2$, then
\begin{equation}\label{eq:cross-coordinate-stability-below-two}
\frac{h_{\Pi_pB_1^n}(a)}
     {h_{\Pi_pB_1^n}\left(a_{(1)}\right)}
\leq
1-c_{p, 1}\|a-a_{(1)}\|_2^2.
\end{equation}
If $p>2$, then
\begin{equation}\label{eq:cross-coordinate-stability-above-two}
\frac{h_{\Pi_pB_1^n}(a)}
     {h_{\Pi_pB_1^n}\left(a_{(1)}\right)}
\geq
1+c_{p, 1}\|a-a_{(1)}\|_2^2.
\end{equation}

For $2 < q < \infty$ there exists a positive constant $c_{p, q}$, depending only on $p$ and $q$, such that if $1 \le p < 2$, then
\begin{align}\label{eq:q-above-two-coordinate-stability-below-two}
    \frac{h_{\Pi_pB_q^n}(a)}{h_{\Pi_pB_q^n}\left(a_{(1)}\right)} \ge 1 + c_{p, q}\left|a - a_{(1)}\right|^2,
\end{align}
and
\begin{align}\label{eq:q-above-two-diagonal-stability-below-two}
    \frac{h_{\Pi_pB_q^n}(a)}{h_{\Pi_pB_q^n}\left(a_{(n)}\right)} \le 1 - c_{p, q}\sum_{i=1}^n\left(a_i^2 - \frac 1 n\right)^2.
\end{align}
If $p > 2$, then
\begin{align}\label{eq:q-above-two-coordinate-stability-above-two}
    \frac{h_{\Pi_pB_q^n}(a)}{h_{\Pi_pB_q^n}\left(a_{(1)}\right)} \le 1 - c_{p, q}\left|a - a_{(1)}\right|^2,
\end{align}
and
\begin{align}\label{eq:q-above-two-diagonal-stability-above-two}
    \frac{h_{\Pi_pB_q^n}(a)}{h_{\Pi_pB_q^n}\left(a_{(n)}\right)} \ge 1 + c_{p, q}\sum_{i=1}^n\left(a_i^2 - \frac 1 n\right)^2.
\end{align}

For $1 < q < 2$ there exists a positive constant $c_{p, q}$, depending only on $p$ and $q$, such that if $1 \le p < 2$, then
\begin{align}\label{eq:q-below-two-coordinate-stability-below-two}
    \frac{h_{\Pi_pB_q^n}(a)}{h_{\Pi_pB_q^n}\left(a_{(1)}\right)} \le 1 - c_{p, q}\left|a - a_{(1)}\right|^{2p}.
\end{align}
If $p > 2$, then
\begin{align}\label{eq:q-below-two-coordinate-stability-above-two}
    \frac{h_{\Pi_pB_q^n}(a)}{h_{\Pi_pB_q^n}\left(a_{(1)}\right)} \ge 1 + c_{p, q}\left|a - a_{(1)}\right|^{2p}.
\end{align}
\end{theorem}

In view of formulas \eqref{eq:lp-cross-polytope-formula} and \eqref{eq:Lp-projection-probabilistic-formula} it follows that $h_{\Pi_pB_q^n}(a)$ is in most cases a $p$-th absolute moment of a weighted sum of certain symmetric independent random variables. Thus some of our stability estimates boil down to the stability of classical inequalities comparing moments, or to their reversals for such sums, i.e. Khintchine-type inequalities. The research on stability of Khintchine-type inequalities is a relatively new topic of research, see \cite{de2016robust, jakimiuk2025khintchine, baranski2025optimal, gao2026fourth} for the results on Rademacher random variables and \cite{jakimiuk2026spheres, chavez2026stability, jakimiuk2026stability} for other distributions. 

In particular, the following corollary is an immediate consequence of \cite[Theorem 2]{jakimiuk2025khintchine} for $p > 3$ and \cite[Theorem 1.4]{gao2026fourth} for $p = 3$. We decided not to include it in the formulation of Theorem \ref{thm:lp-projection-stability} since the whole work of proving it is done elsewhere and because it does not cover the interval $2 < p < 3$.

\begin{corollary}\label{cor:cross-diagonal-stability-above-three}
    For $p \ge 3$ there exists a constant $C_{p, 1} > 0$ such that for any $a \in \Sp^{n-1}$ the stability estimate
    \begin{align*}
        \frac{h_{\Pi_pB_1^n(a)}}{h_{\Pi_pB_1^n\left(a_{(n)}\right)}} \le 1 - C_{p, 1}\sum_{i=1}^n\left(a_i^2 - \frac 1 n\right)^2
    \end{align*}
    holds.
\end{corollary}

We want to emphasize that \eqref{eq:cross-coordinate-stability-below-two} and \eqref{eq:cross-coordinate-stability-above-two} yield coordinate stability estimates for the classical moment comparison inequalities for Rademacher sums. Indeed, for $p\geq1$, using the elementary inequalities $(1+x)^p \geq1+xp$, for $x\geq 0$ and $(1-x)^p \leq1-x$, for $0<x<1$, we obtain the following reformulation of \eqref{eq:cross-coordinate-stability-below-two} and \eqref{eq:cross-coordinate-stability-above-two}.
\begin{corollary}\label{cor:coordinate-rademacher-stability}
Let \(a\in \Sp^{n-1}\) with $a_1 \ge a_2 \ge \ldots \ge a_n \ge 0$ and $c_{p, 1}>0$ be the constant from Theorem \ref{thm:lp-projection-stability}. For \(1\leq p<2\),
$$
\E\left|\sum_{j=1}^n a_j\e_j\right|^p
\leq
1-c_{p, 1}\|a-a_{(1)}\|_2^2.
$$
For \(p>2\),
$$
\E\left|\sum_{j=1}^n a_j\e_j\right|^p
\geq
1+pc_{p, 1}\|a-a_{(1)}\|_2^2.
$$
\end{corollary}

It turns out that for $q > 2$ the random variables from Proposition \ref{prop:Lp-projection-formula} are so-called \textit{Gaussian mixtures}, as observed by Eskenazis, Nayar and Tkocz in \cite{eskenazis2018gaussian}. We recall the definition.

\begin{defn}\label{def:gaussian-mixtures}
    A random variable $X$ is called a Gaussian mixture if there exists a positive random variable $Y$ and a standard Gaussian random variable $Z$, independent of $Y$, such that $X$ has the same distribution as the product $YZ$.
\end{defn}

In Section \ref{sec:mixtures} we recall the majorization inequality of Eskenazis, Nayar and Tkocz for Gaussian mixtures and prove its stability refinement, which will be one of the crucial tools in the proof of Theorem \ref{thm:lp-projection-stability}.

\subsection{Monotonicity of support ratios for $L_p$-projection bodies}

It is natural to ask whether the monotonicity of normalized
hyperplane projection volumes established by Barthe and Naor
\cite[Theorem~3]{barthe2002hyperplane} extends to support ratios of
$L_p$-projection bodies.

\begin{theorem}\label{thm:monotonicity-normalized-lp-projections}
Let $n\geq 2$, $1\leq p<\infty$, and $a\in \Sp^{n-1}$.
If $1\leq q_1<q_2\leq\infty$, then
\[
\frac{h_{\Pi_p B_{q_1}^n}(a)}
     {h_{\Pi_p B_{q_1}^n}(e_1)}
\leq
\frac{h_{\Pi_p B_{q_2}^n}(a)}
     {h_{\Pi_p B_{q_2}^n}(e_1)},
\qquad 1\leq p<2,
\]
whereas
\[
\frac{h_{\Pi_p B_{q_1}^n}(a)}
     {h_{\Pi_p B_{q_1}^n}(e_1)}
\geq
\frac{h_{\Pi_p B_{q_2}^n}(a)}
     {h_{\Pi_p B_{q_2}^n}(e_1)},
\qquad p>2.
\]
For $p=2$, both ratios are equal to $1$.
\end{theorem}
Taking $p=1$ in Theorem~\ref{thm:monotonicity-normalized-lp-projections},
we recover the monotonicity theorem.

\begin{corollary}[Barthe--Naor]\label{cor:barthe-naor-monotonicity}
Let $n\geq2$ and $a\in \Sp^{n-1}$. Then the function
\[
q\longmapsto
\frac{\operatorname{vol}_{n-1}(\operatorname{Proj}_{a^\perp}B_q^n)}
     {\operatorname{vol}_{n-1}(B_q^{n-1})}
\]
is nondecreasing on $[1,\infty]$.
\end{corollary}

\subsection{Organization of the paper} 

In Section 2 we prove the stability estimates for the regular simplex. In Section \ref{sec:3} we prove the probabilistic formula for $L_p$-projection bodies claimed in Subsection 1.3. In Section 4 we establish our results for Gaussian mixtures. In Section 5 we prove the stability estimates for $L_p$-projection bodies and, finally, in Section 6 we prove the generalization of Barthe's and Naor's monotonicity result to $L_p$-projection bodies. 

\section{Proof of Theorem \ref{thm:simplex-projection-stability}}
We begin with the lower stability estimate. Set
\[
S
=
\sum_{a_k>0}a_k
=
-\sum_{a_k<0}a_k
=
\frac12\sum_{k=1}^{n+1}|a_k|,
\]
and choose $i$ and $j$ so that
$a_i=\max_{1\leq k\leq n+1}a_k$ and
$a_j=\min_{1\leq k\leq n+1}a_k$, respectively. Notice that $a_i>0$
and $a_j<0$, since there must be at least one positive and one negative
coordinate.

By \eqref{eq:simplex-projection-formula},
\begin{equation}\label{eq:minimum-volume-ratio}
\frac{
\operatorname{vol}_{n-1}
\left(
\operatorname{Proj}_{a^\perp\cap\mathcal H}\Delta_n
\right)
}{
\operatorname{vol}_{n-1}
\left(
\operatorname{Proj}_{a_\ast^\perp\cap\mathcal H}\Delta_n
\right)
}
=
\sqrt2S.
\end{equation}

Since $a_i$ is the largest positive coordinate and $-a_j$ is the
largest absolute value among the negative coordinates,
\[
1
=
\sum_{a_k>0}a_k^2+
\sum_{a_k<0}a_k^2
\leq
a_iS-a_jS.
\]
Consequently,
\[
a_i-a_j\geq\frac1S.
\]
Moreover, $a_i\leq S$ and $-a_j\leq S$, so
\[
1\leq S(a_i-a_j)\leq2S^2.
\]
Thus, $S\geq1/\sqrt2$.

Since
\[
\sum_{k\neq i,j}|a_k|
=
2S-(a_i-a_j),
\]
and $S\geq1/\sqrt2$, we have
\[
\left|a_i-\frac1{\sqrt2}\right|
\leq
S-\frac1{\sqrt2}+S-a_i
\]
and
\[
\left|a_j+\frac1{\sqrt2}\right|
\leq
S-\frac1{\sqrt2}+S+a_j.
\]
The triangle inequality now gives
\begin{align*}
\min_{k\neq\ell}
\left\|
a-\frac{e_k-e_\ell}{\sqrt2}
\right\|_2
&\leq
\left\|
a-\frac{e_i-e_j}{\sqrt2}
\right\|_2\\
&\leq
\left|a_i-\frac1{\sqrt2}\right|
+
\left|a_j+\frac1{\sqrt2}\right|
+
\sum_{k\neq i,j}|a_k|\\
&\leq
2\left(S-\frac1{\sqrt2}\right)
+2\bigl(2S-(a_i-a_j)\bigr)\\
&\leq
2\left(S-\frac1{\sqrt2}\right)
+2\left(2S-\frac1S\right).
\end{align*}
Now
\[
2\left(S-\frac1{\sqrt2}\right)
=
\sqrt2(\sqrt2S-1),
\]
while
\begin{align*}
2\left(2S-\frac1S\right)
&=
2(\sqrt2S-1)\frac{\sqrt2S+1}{S}\\
&\leq
4\sqrt2(\sqrt2S-1),
\end{align*}
because $S\geq1/\sqrt2$. Therefore,
\[
\min_{k\neq\ell}
\left\|
a-\frac{e_k-e_\ell}{\sqrt2}
\right\|_2
\leq
5\sqrt2(\sqrt2S-1)
\leq
8(\sqrt2S-1).
\]
Hence,
\[
\sqrt2S-1
\geq
\frac18
\min_{k\neq\ell}
\left\|
a-\frac{e_k-e_\ell}{\sqrt2}
\right\|_2.
\]
Together with \eqref{eq:minimum-volume-ratio}, this proves
\eqref{eq:sharp-lower-stability}.

We next prove the upper stability estimates. Suppose first that $n$ is
odd and write $n+1=2h$. Let $r$ be the number of positive coordinates
of $a$. Thus, $r$ is determined by $a$ and, by our normalization,
$1\leq r\leq h$. Since the coordinates are decreasing,
\[
\sum_{j=1}^{h}a_j
= \sum_{j=1}^ra_j+\sum_{j=r+1}^ha_j=
S-\sum_{j=r+1}^{h}|a_j|,
\]
where the sum on the right is understood to be zero when $r=h$.

By \eqref{eq:simplex-projection-formula},
\[
\frac{
\operatorname{vol}_{n-1}
\left(
\operatorname{Proj}_{a^\perp\cap\mathcal H}\Delta_n
\right)
}{
\operatorname{vol}_{n-1}
\left(
\operatorname{Proj}_{a_{\mathrm{odd}}^\perp\cap\mathcal H}\Delta_n
\right)
}
=
\frac{2S}{\sqrt{2h}}.
\]
Moreover,
\begin{align*}
\frac12\|a-a_{\mathrm{odd}}\|_2^2
&=
1-\langle a,a_{\mathrm{odd}}\rangle\\
&=
1-\frac{2S}{\sqrt{2h}}
+\frac2{\sqrt{2h}}
\sum_{j=r+1}^{h}|a_j|.
\end{align*}
If $r=h$, the last sum vanishes and the desired estimate follows
immediately. Suppose that $r<h$. Since
$\sum_{j=r+1}^{2h}|a_j|=S$ and
$|a_{r+1}|\leq\cdots\leq|a_{2h}|$, the numbers
$|a_{r+1}|,\ldots,|a_h|$ are the $h-r$ smallest among the
$2h-r$ numbers $|a_{r+1}|,\ldots,|a_{2h}|$. Hence, their average is
no larger than the average of all these numbers, i.e.,
\[
\frac{\displaystyle\sum_{j=r+1}^{h}|a_j|}{h-r}
\leq
\frac{\displaystyle\sum_{j=r+1}^{2h}|a_j|}{2h-r}
=\frac{S}{2h-r}.
\]
Hence
\begin{equation}\label{eq:odd-small-coordinates}
\sum_{j=r+1}^{h}|a_j|
\leq
\frac{h-r}{2h-r}S.
\end{equation}
The Cauchy--Schwarz inequality also gives
\[
1
\geq
\frac{S^2}{r}+\frac{S^2}{2h-r}.
\]
Consequently,
\[
\left(\frac{2S}{\sqrt{2h}}\right)^2
\leq
\frac{r(2h-r)}{h^2}
=
1-\frac{(h-r)^2}{h^2},
\]
and therefore, since $\frac{2S}{\sqrt{2h}}\leq 1$,
\begin{equation}\label{eq:odd-volume-deficit}
1-\frac{2S}{\sqrt{2h}}
\geq
\frac{(h-r)^2}{2h^2}.
\end{equation}
Here, in the last inequality, we used the fact that
$1-x=\frac{1-x^2}{1+x}\geq\frac{1-x^2}{2}$ for $0\leq x\leq1$.

On the other hand, \eqref{eq:odd-small-coordinates} implies
\begin{align*}
\frac2{\sqrt{2h}}
\sum_{j=r+1}^{h}|a_j|
&\leq
\frac{h-r}{2h-r}\frac{2S}{\sqrt{2h}}\\
&\leq
\frac{h-r}{h}\\
&\leq
2h\left(1-\frac{2S}{\sqrt{2h}}\right),
\end{align*}
where the last inequality follows from
\eqref{eq:odd-volume-deficit}. It follows that
\[
\frac12\|a-a_{\mathrm{odd}}\|_2^2
\leq
(2h+1)\left(1-\frac{2S}{\sqrt{2h}}\right)
=
(n+2)\left(1-\frac{2S}{\sqrt{2h}}\right).
\]
This proves \eqref{eq:odd-optimal-stability}.

Suppose now that $n$ is even and write $n=2h$. Once again, let $r$ be
the number of positive coordinates of $a$, so $1\leq r\leq h$. As in
the odd case,
\[
\sum_{j=1}^{h}a_j
=
S-\sum_{j=r+1}^{h}|a_j|.
\]
By an easy computation,
\[
\sqrt{\frac{h+1}{h(2h+1)}}
+
\sqrt{\frac{h}{(h+1)(2h+1)}}
=
\sqrt{\frac{2h+1}{h(h+1)}}.
\]
Therefore, \eqref{eq:simplex-projection-formula} gives
\[
\frac{
\operatorname{vol}_{n-1}
\left(
\operatorname{Proj}_{a^\perp\cap\mathcal H}\Delta_n
\right)
}{
\operatorname{vol}_{n-1}
\left(
\operatorname{Proj}_{a_{\mathrm{even}}^\perp\cap\mathcal H}\Delta_n
\right)
}
=
\sqrt{\frac{2h+1}{h(h+1)}}S,
\]
while, from the definition of $a_{\mathrm{even}}$,
\begin{align*}
\frac12\|a-a_{\mathrm{even}}\|_2^2
&=
1-\langle a,a_{\mathrm{even}}\rangle\\
&=
1-\sqrt{\frac{2h+1}{h(h+1)}}S
+\sqrt{\frac{2h+1}{h(h+1)}}
\sum_{j=r+1}^{h}|a_j|.
\end{align*}
If $r=h$, the desired estimate again follows immediately. Suppose
that $r<h$. Since there are $2h+1-r$ nonpositive coordinates,
\begin{equation}\label{eq:even-small-coordinates}
\sum_{j=r+1}^{h}|a_j|
\leq
\frac{h-r}{2h+1-r}S.
\end{equation}
The Cauchy--Schwarz inequality gives
\[
1
\geq
\frac{S^2}{r}+\frac{S^2}{2h+1-r},
\]
and consequently
\[
\frac{2h+1}{h(h+1)}S^2
\leq
\frac{r(2h+1-r)}{h(h+1)}
=
1-\frac{(h-r)(h-r+1)}{h(h+1)}.
\]
It follows, as before, that
\begin{equation}\label{eq:even-volume-deficit}
1-\sqrt{\frac{2h+1}{h(h+1)}}S
\geq
\frac{(h-r)(h-r+1)}{2h(h+1)}.
\end{equation}
Moreover, \eqref{eq:even-small-coordinates} gives
\begin{align*}
\sqrt{\frac{2h+1}{h(h+1)}}
\sum_{j=r+1}^{h}|a_j|
&\leq
\frac{h-r}{2h+1-r}
\sqrt{\frac{2h+1}{h(h+1)}}S\\
&\leq
\frac{h-r}{h+1}\\
&\leq
h\left(
1-\sqrt{\frac{2h+1}{h(h+1)}}S
\right),
\end{align*}
where the last inequality follows from
\eqref{eq:even-volume-deficit}. Therefore,
\[
\frac12\|a-a_{\mathrm{even}}\|_2^2
\leq
(h+1)
\left(
1-\sqrt{\frac{2h+1}{h(h+1)}}S
\right).
\]
Since $n=2h$, this proves \eqref{eq:even-optimal-stability}.

Finally, the exponent $1$ in the lower stability estimate is optimal.
It suffices to consider
$$
a(t)=\frac{1}{\sqrt{2(1-t+t^2)}}(1,-(1-t),-t,0,\ldots,0).
$$
Then, $a(t)\in \mathbb{S}^n\cap \mathcal{H}'$ and
$0<d(t):=\min_{i\neq j}\left\|a(t)-\frac{e_i-e_j}{\sqrt{2}}\right\|
\to 0$ as $t\to 0^+$. One can directly compute
\[
\frac{\operatorname{vol}_{n-1}\left(\operatorname{Proj}_{a(t)^{\perp}\cap \mathcal{H}'}\Delta_n\right)}
{\operatorname{vol}_{n-1}\left(\operatorname{Proj}_{a_*^{\perp}\cap \mathcal{H}'}\Delta_n\right)}
-1
=\frac{1}{\sqrt{1-t+t^2}}-1
=\frac{t}{2}+O(t^2).
\]
Moreover,
\begin{equation}\label{d^2}
d(t)^2=2-\frac{2-t}{\sqrt{1-t+t^2}}
=\frac{3}{4}t^2+O(t^3),
\end{equation}
consequently,
\[
d(t)=\frac{\sqrt{3}}{2}t+O(t^2).
\]
Since $O(d^2(t))=O(t^2)$, from equation~\eqref{d^2}, we get
$t=\frac{2}{\sqrt{3}}d(t)+O(d^2(t))$.
Finally, notice that $O(d^2(t))=o(d(t))$, and thus
\[
\frac{\operatorname{vol}_{n-1}\left(\operatorname{Proj}_{a(t)^{\perp}\cap \mathcal{H}'}\Delta_n\right)}
{\operatorname{vol}_{n-1}\left(\operatorname{Proj}_{a_*^{\perp}\cap \mathcal{H}'}\Delta_n\right)}
-1
=\frac{1}{\sqrt{3}}d(t)+o(d(t)).
\]
If the estimate held with an exponent $0<\alpha<1$, then, for some $c>0$,
\[
c\,d(t)^\alpha
\leq
\frac{
\operatorname{vol}_{n-1}
\left(
\operatorname{Proj}_{a(t)^\perp\cap\mathcal H'}\Delta_n
\right)
}{
\operatorname{vol}_{n-1}
\left(
\operatorname{Proj}_{a_*^\perp\cap\mathcal H'}\Delta_n
\right)
}-1
=\frac{1}{\sqrt{3}}d(t)+o(d(t)).
\]
Dividing by $d(t)^\alpha$ and letting $t\to0^+$ gives
$c\leq0$, a contradiction.

Similarly, the exponent 2 in the upper stability estimate is optimal. To see this, let $a_0$ be $a_{\text{even}}$ or $a_{odd}$, depending on the parity of $n$, and consider the curve, for $t>0$,
\[
a(t)=\frac{1}{\sqrt{1+t^2}}\left(a_0+t\frac{e_i-e_j}{\sqrt{2}}\right),
\]
for $i,j$ in the same block of $a$. Indeed, $\|a(t)-a_0\|_2\to 0$ and
\[
1-\frac{\operatorname{vol}_{n-1}\left(\operatorname{Proj}_{a(t)^{\perp}\cap \mathcal{H}'}\Delta_n\right)}
{\operatorname{vol}_{n-1}\left(\operatorname{Proj}_{a_0^{\perp}\cap \mathcal{H}'}\Delta_n\right)}=1-\langle a(t),a_0\rangle=\frac{1}{2}\|a(t)-a_0\|_2^2.
\]
If either upper stability estimate held with an exponent $0<\alpha<2$, and a positive constant, we would obtain a
contradiction.

\section{Probabilistic formula for $L_p$-projection bodies}\label{sec:3}

Using the notation and definitions from the introduction, we
derive a probabilistic formula for the support functions of the
$L_p$-projection bodies of $\ell_q$-balls. These formulas will
be used to prove the stability estimates and the monotonicity
of normalized support ratios in the following sections.
For background on the $L_p$-Brunn--Minkowski theory, we refer to
\cite{schneider2013convex,gardner2006geometric}.

For $p = 1$ a probabilistic formula for $h_{\Pi_pB_q^n}$ for $1 < q < \infty$ was obtained by Barthe and Naor in \cite{barthe2002hyperplane}. Now we extend it to $p > 1$. The proof below follows closely the strategy of \cite{barthe2002hyperplane}.
\begin{proof}[Proof of Proposition~\ref{prop:Lp-projection-formula}]
    We start by transferring the integration from $\Sp^{n-1}$ to $\partial K$ and then replacing the surface area measure of $K$ with the more convenient cone measure. For a convex body $K \subset \R^n$ with $0 \in \text{int}K$ let $\sigma_K$ be the surface area measure of $\partial K$ normalized to be a probability measure, so that $|\partial K|d\sigma_K(x) = dS_{n-1}(K, \n(x))$, where $|\partial K|$ is the classical surface area of $K$ and $\n(x)$ is the outer unit normal vector to $\partial K$ at $x$. Let also $\mu_K$ be the normalized cone measure on $\partial K$, i.e. $\mu_K(B) = \frac{\vol_n(\text{cone}(B))}{\vol_n(K)}$, where $B \subseteq \partial K$ and $\text{cone}(B) = \{tb \colon 0 \le t \le 1, b \in B\}$. It was proved in \cite{naor2003projecting} that for almost every $x \in \partial K$ we have
    \begin{align}\label{diffmes}
        \frac{d\sigma_K}{d\mu_K}(x) = \frac{n\vol_n(K)}{|\partial K|}\left\|\nabla(\|\cdot\|_K)(x)\right\|_2,
    \end{align}
    where $\|\cdot\|_K$ is the Minkowski functional of $K$. Note that $\n(x) = \frac{\nabla(\|\cdot\|_K)(x)}{\|\nabla(\|\cdot\|_K)(x)\|_2}$ and the function $h_K$ is one-homogeneous. Using the notation above, these observations and identity \eqref{diffmes} we have
    \begin{align*}
        h_{\Pi_pK}^p(a) &= \frac{|\partial K|}{2}\int_{\partial K}|\langle a, \n(x)\rangle|^ph_K(\n(x))^{1-p}d\sigma_K(x) \\ &= \frac{|\partial K|}{2}\int_{\partial K}\frac{|\langle a, \nabla(\|\cdot\|_K)(x)\rangle|^p}{\|\nabla(\|\cdot\|_K)(x)\|_2^p} \cdot \frac{h_K(\nabla(\|\cdot\|_K)(x))^{1-p}}{\|\nabla(\|\cdot\|_K)(x)\|_2^{1-p}}\frac{d\sigma_K}{d\mu_K}(x)d\mu_K(x) \\ &= \frac{n\vol_n(K)}{2}\int_{\partial K}|\langle a, \nabla(\|\cdot\|_K)(x)\rangle|^ph_K^{1-p}(\nabla(\|\cdot\|_K)(x))d\mu_K(x).
    \end{align*}
    Now, let $K = B_q^n$. Since $\|x\|_{B_q^n} = \|x\|_q$ and $h_{B_q^n}(x) = \|x\|_{q^*}$, where $q^*$ is the H\"older conjugate of $q$, we get
    \begin{align*}
        h_{\Pi_pB_q^n}^p(a) = \frac{nq\vol_n(B_q^n)}{2}\int_{\partial B_q^n}\left|\sum_{i=1}^na_i|x_i|^{q-1}\text{sign}(x_i)\right|^p \cdot \left|\sum_{i=1}^n|x_i|^{(q-1)q^*}\right|^{\frac{1-p}{q^*}}d\mu_{B_q^n}(x).
    \end{align*}
    Let $Y_1, \ldots, Y_n$ be i.i.d. random variables with density proportional to $e^{-|t|^q}$, $S = \left(\sum_{i=1}^n|Y_i|^q\right)^{\frac 1 q}$ and $Z = \left(\frac{Y_1}{S}, \ldots, \frac{Y_n}{S}\right)$. It was proved in the independent works \cite{schechtman1990volume, rachev1991approximate} that $S$, $Z$ are independent and $\mu_{B_q^n}$ is the distribution of $Z$. Thus we get
    \begin{align*}
        h_{\Pi_pB_q^n}^p(a) &= \frac{nq\vol_n(B_q^n)}{2}\E_Z\left[\frac{\left|\sum_{i=1}^na_i|Y_i|^{q-1}\text{sign}(Y_i)\right|^p}{S^{p(q-1)}} \cdot \frac{\left|\sum_{i=1}^n|Y_i|^q\right|^{\frac{(q-1)(1-p)}{q}}}{S^{(1-p)(q-1)}}\right] \\ &= \frac{nq\vol_n(B_q^n)}{2} \cdot \frac{\E S^{p(q-1)}}{\E S^{p(q-1)}}\E_Z\left[\frac{\left|\sum_{i=1}^na_i|Y_i|^{q-1}\text{sign}(Y_i)\right|^p}{S^{p(q-1)}} \cdot \frac{S^{(1-p)(q-1)}}{S^{(1-p)(q-1)}}\right] \\ &= \frac{nq\vol_n(B_q^n)}{2\E S^{p(q-1)}}\E\left|\sum_{i=1}^na_i|Y_i|^{q-1}\text{sign}(Y_i)\right|^p.
    \end{align*}
    To finish the proof with $C_{p, q, n} = \frac{n\vol_n(B_q^n)}{2\E S^{p(q-1)}}$ it remains to observe that $|Y_i|^{q-1}\text{sign}(Y_i)$ has the same distribution as $X_i$, which is a simple exercise.
\end{proof}

\section{Gaussian mixtures}\label{sec:mixtures}

As mentioned in the introduction, it follows from the formula \eqref{eq:Lp-projection-probabilistic-formula} that proving the stability estimates for $\Pi_pB_q^n$ for $1 < q < \infty$ boils down to showing suitable stability estimates for the random variables $X_i$ from Proposition \ref{prop:Lp-projection-formula}. Moreover, these random variables are Gaussian mixtures.

Before presenting our results for Gaussian mixtures we shall recall the notion of the Schur order. We say that a vector $a = (a_1, \ldots, a_n)$ is dominated by a vector $b = (b_1, \ldots, b_n)$ in the Schur order (and denote it by $a \preceq b$) if
\begin{align*}
    \sum_{i=1}^na_i = \sum_{i=1}^nb_i \quad \text{and} \quad \forall_{1 \le k < n}\sum_{i=1}^ka_i^* \le \sum_{i=1}^kb_i^*,
\end{align*}
where $a_1^*, \ldots, a_n^*$ and $b_1^*, \ldots, b_n^*$ are the non-increasing rearrangements of the coordinates of $a$ and $b$, respectively. Theorem 3 in \cite{eskenazis2018gaussian} (see also \cite[Proposition 2.6]{avercamp2003wavelet}) shows that if $X_1, \ldots, X_n$ are i.i.d. Gaussian mixtures, then the function $\left(a_1^2, \ldots, a_n^2\right) \mapsto \E\left|\sum_{i=1}^na_iX_i\right|^p$ is decreasing with respect to the Schur order for $p < 2$ and increasing for $p > 2$. We shall prove the following strengthening of this result, which also generalizes \cite[Theorem 2]{jakimiuk2026stability} to $p > 1$.

\begin{theorem}\label{thm:gaussian-mixtures}
    Let $X_1, \ldots, X_n$ be i.i.d. square integrable Gaussian mixtures whose distribution is not Gaussian. Then for any $1 \le p < 2$ there exists a positive constant $\kappa_p$, depending only on the distribution of $X_1$ and on $p$, such that for any real numbers $a_1, \ldots, a_n$, $b_1, \ldots, b_n$ satisfying $\sum_{i=1}^na_i^2 = \sum_{i=1}^nb_i^2 = 1$ and $\left(b_1^2, \ldots, b_n^2\right) \preceq \left(a_1^2, \ldots, a_n^2\right)$ we have
    \begin{align*}
        \E\left|\sum_{i=1}^nb_iX_i\right|^p - \E\left|\sum_{i=1}^na_iX_i\right|^p \ge \kappa_p\left(\sum_{i=1}^na_i^4 - \sum_{i=1}^nb_i^4\right).
    \end{align*}
    Moreover, if additionally $\E|X_1|^p < \infty$, then 
    \begin{align*}
        \E\left|\sum_{i=1}^na_iX_i\right|^p - \E\left|\sum_{i=1}^nb_iX_i\right|^p \ge \kappa_p\left(\sum_{i=1}^na_i^4 - \sum_{i=1}^nb_i^4\right).
    \end{align*}
    holds for $2 < p < \infty$.
\end{theorem}

Theorem \ref{thm:gaussian-mixtures} is crucial in the proof of \eqref{eq:q-above-two-coordinate-stability-below-two}, \eqref{eq:q-above-two-diagonal-stability-below-two}, \eqref{eq:q-above-two-coordinate-stability-above-two}, and \eqref{eq:q-above-two-diagonal-stability-above-two}, we also believe that it is of independent interest.

We shall need the following auxiliary lemma.

\begin{lemma}\label{lem:gaussian-mixtures-auxiliary}
    Let $Y_1$, $Y_2$ be i.i.d. square integrable positive random variables whose distribution is not a Dirac delta distribution. Let also $a_1$, $a_2$, $b_1$, $b_2$, $s$ be real numbers such that $s \ge 0$, $a_1^2 + a_2^2 = b_1^2 + b_2^2 = \sigma^2 > 0$ and $a_1^2 \ge b_1^2 \ge b_2^2 \ge a_2^2$. Let $\varphi \colon [0, \infty) \to \R$ be a convex function with the second derivative monotone on $(0, \infty)$. Moreover, assume that $\varphi\left(x_1Y_1^2 + x_2Y_2^2 + z\right)$ is integrable for every $x_1, x_2, z \in \R$. Then there exist positive constants $c$, $\eta$, depending only on $\varphi$ and the distribution of $Y_1$, such that
    \begin{align*}
        \E\varphi\left(a_1^2Y_1^2 + a_2^2Y_2^2 + s\right) - \E\varphi\left(b_1^2Y_1^2 + b_2^2Y_2^2 + s\right) \ge c\varphi''\left(s + \eta\sigma^2\right)\left(a_1^4 + a_2^4 - b_1^4 - b_2^4\right).
    \end{align*}
\end{lemma}

\begin{proof}[Proof of Lemma \ref{lem:gaussian-mixtures-auxiliary}]
    Denote $\Delta = \E\varphi\left(a_1^2Y_1^2 + a_2^2Y_2^2 + s\right) - \E\varphi\left(b_1^2Y_1^2 + b_2^2Y_2^2 + s\right)$ and define numbers $\rho \ge \mu \ge \frac 1 2$ such that $a_1^2 = \rho\sigma^2$ and $b_1^2 = \mu\sigma^2$. Define also random variables $R = \frac{Y_1^2 + Y_2^2}{2}$ and $D = Y_1^2 - Y_2^2$. Note that $D$ is a symmetric random variable and $R$ is independent of $\sgn(D)$, hence $(R, D)$ has the same distribution as $(R, |D|\e)$, where $\e$ is a Rademacher random variable independent of $(Y_1, Y_2)$. We have
    \begin{align*}
        \E\varphi\left(a_1^2Y_1^2 + a_2^2Y_2^2 + s\right) &= \E\varphi\left(s + \sigma^2R + \left(\rho - \frac 1 2\right)\sigma^2D\right) \\ &= \E_{\e}\E_{R, D}\varphi\left(s + \sigma^2R + \left(\rho - \frac 1 2\right)\sigma^2|D|\e\right) \\ &= \frac 1 2\E\left[\varphi\left(s + \sigma^2R + \left(\rho - \frac 1 2\right)\sigma^2|D|\right) + \varphi\left(s + \sigma^2R - \left(\rho - \frac 1 2\right)\sigma^2|D|\right)\right]
    \end{align*}
    and in the same manner
    \begin{align*}
        \E\varphi\left(b_1^2Y_1^2 + b_2^2Y_2^2 + s\right) = \frac 1 2\E\left[\varphi\left(s + \sigma^2R + \left(\mu - \frac 1 2\right)\sigma^2|D|\right) + \varphi\left(s + \sigma^2R - \left(\mu - \frac 1 2\right)\sigma^2|D|\right)\right]
    \end{align*}
    holds. By subtracting and using basic calculus we get
    \begin{align*}
        \Delta &= \frac 1 2\E\left[\int_{\sigma^2|D|\left(\mu - \frac 1 2\right)}^{\sigma^2|D|\left(\rho - \frac 1 2\right)}\varphi'\left(s + \sigma^2R + t\right)dt - \int_{-\sigma^2|D|\left(\rho - \frac 1 2\right)}^{-\sigma^2|D|\left(\mu - \frac 1 2\right)}\varphi'\left(s + \sigma^2R + t\right)dt\right] \\ &= \frac 1 2 \E\int_{\sigma^2|D|\left(\mu - \frac 1 2\right)}^{\sigma^2|D|\left(\rho - \frac 1 2\right)}\left(\varphi'\left(s + \sigma^2R + t\right) - \varphi'\left(s + \sigma^2R - t\right)\right)dt \\ &= \frac{\sigma^2}{2}\E|D|\int_{\mu - \frac 1 2}^{\rho - \frac 1 2}\left(\varphi'\left(s + \sigma^2R + \sigma^2|D|u\right) - \varphi'\left(s + \sigma^2R - \sigma^2|D|u\right)\right)du \\ &= \frac{\sigma^2}{2}\E|D|\int_{\mu - \frac 1 2}^{\rho - \frac 1 2}\int_{-\sigma^2|D|u}^{\sigma^2|D|u}\varphi''\left(s + \sigma^2R + v\right)dvdu \\ &= \frac{\sigma^4}{2}\E|D|^2\int_{\mu - \frac 1 2}^{\rho - \frac 1 2}\int_{-u}^u\varphi''\left(s + \sigma^2R + \sigma^2|D|w\right)dwdu.
    \end{align*}
    If $\varphi''$ is non-increasing, then the bound $\varphi''\left(s + \sigma^2R + \sigma^2|D|w\right) \ge \varphi''\left(s + \sigma^2R\right)$ holds for $-u \le w \le 0$. Moreover, since the distribution of $Y_1$ is not a Dirac delta distribution and $Y_2$ is an independent copy of $Y_1$, we have that $\pp(|D| \ge \delta, R \le \eta) > 0$ for some $\delta, \eta > 0$. Thus, using positivity of $\varphi''$, we bound
    \begin{align*}
        \Delta &\ge \frac{\sigma^4}{2}\E\1_{|D| \ge \delta, R \le \eta}|D|^2\int_{\mu - \frac 1 2}^{\rho - \frac 1 2}\int_{-u}^0\varphi''\left(s + \sigma^2R\right)dwdu \\ &\ge \frac{\sigma^4}{2}\pp(|D| \ge \delta, R \le \eta)\delta^2\varphi''\left(s + \eta\sigma^2\right)\int_{\mu - \frac 1 2}^{\rho - \frac 1 2}udu \\ &= c\varphi''\left(s + \eta\sigma^2\right)\sigma^4\left(\left(\rho - \frac 1 2\right)^2 - \left(\mu - \frac 1 2\right)^2\right) \\ &= c\varphi''\left(s + \eta\sigma^2\right)\sigma^4\frac{\rho^2 + (1 - \rho)^2 - \mu^2 - (1 - \mu)^2}{2} = \frac 1 2c\varphi''\left(s + \eta\sigma^2\right)\left(a_1^4 + a_2^4 - b_1^4 - b_2^4\right),
    \end{align*}
    which completes the proof. For $\varphi''$ non-decreasing we proceed in the same way, only replacing $-u \le w \le 0$ with $0 \le w \le u$ and $R \le \eta$ with $R \ge \eta$.
\end{proof}

Now we are ready to prove Theorem \ref{thm:gaussian-mixtures}.

\begin{proof}[Proof of Theorem \ref{thm:gaussian-mixtures}]
    We start in a manner similar to the proof of Theorem 3 in \cite{eskenazis2018gaussian}. Let $Y_1, \ldots, Y_n, Z_1, \ldots, Z_n$, $Z$ be random variables such that $Y_i$ are positive, $Z$, $Z_i$ are standard Gaussians, $X_i$ has the same distribution as $Y_iZ_i$ for all $i$, and variables $X_1, \ldots, X_n$, $Y_1, \ldots, Y_n$, $Z_1, \ldots, Z_n$, $Z$ are independent. Then we have
    \begin{align*}
        \E\left|\sum_{i=1}^na_iX_i\right|^p = \E\left|\sum_{i=1}^na_iY_iZ_i\right|^p = \E_Y\E_Z\left|\left(\sum_{i=1}^na_i^2Y_i^2\right)^{\frac 1 2}Z\right|^p = \E|Z|^p \cdot \E\left|\sum_{i=1}^na_i^2Y_i^2\right|^{\frac p 2}.
    \end{align*}
    Since $\E|Z|^p$ is a constant depending only on $p$, it suffices to prove that
    \begin{align*}
        \E\left|\sum_{i=1}^nb_i^2Y_i^2\right|^{\frac p 2} - \E\left|\sum_{i=1}^na_i^2Y_i^2\right|^{\frac p 2} \ge \kappa_p\left(\sum_{i=1}^na_i^4 - \sum_{i=1}^nb_i^4\right)
    \end{align*}
    for $1 \le p < 2$ and
    \begin{align*}
        \E\left|\sum_{i=1}^na_i^2Y_i^2\right|^{\frac p 2} - \E\left|\sum_{i=1}^nb_i^2Y_i^2\right|^{\frac p 2} \ge \kappa_p\left(\sum_{i=1}^na_i^4 - \sum_{i=1}^nb_i^4\right)
    \end{align*}
    for $p > 2$. Due to technical reasons we shall divide the rest of the proof into three cases.

    \textit{Case $1 \le p < 2$.} By a known characterization of the Schur order we have $\left(b_1^2, \ldots, b_n^2\right) = b \preceq a = \left(a_1^2, \ldots, a_n^2\right)$ if and only if $b$ can be obtained from $a$ through a composition of finitely many $T$-transformations, where a $T$-transformation is a transformation bringing two coefficients of a vector closer together while preserving their sum and all remaining coefficients. Denote these $T$-transformations by $T_1, \ldots, T_k$, so that $b = T_kT_{k-1} \ldots T_1a$, and denote $T^j = T_jT_{j-1} \ldots T_1$ for $j = 1, 2, \ldots, k$. Then, by a telescoping sum argument, it suffices to prove that
    \begin{align}\label{eq:gaussian-lemma-t-transpositions}
        \E\left|\sum_{i=1}^n\left(T^ja\right)_iY_i^2\right|^{\frac p 2} - \E\left|\sum_{i=1}^n\left(T^{j-1}a\right)_iY_i^2\right|^{\frac p 2} \ge \kappa_p\left(\sum_{i=1}^n\left(T^{j-1}a\right)_i^2- \sum_{i=1}^n\left(T^ja\right)_i^2\right)
    \end{align}
    for $j = 1, \ldots, k$. By invariance under permutations and the structure of $T$-transformations, in order to simplify the notation we shall slightly abuse it and write that $T^{j-1}a = \left(a_1^2, \ldots, a_n^2\right)$, $T^ja = \left(b_1^2, b_2^2, a_3^2, \ldots, a_n^2\right)$ and $S = \sum_{i=3}^na_i^2Y_i^2$. Then \eqref{eq:gaussian-lemma-t-transpositions} becomes
    \begin{align*}
        \E\left|b_1^2Y_1^2 + b_2^2Y_2^2 + S\right|^{\frac p 2} - \E\left|a_1^2Y_1^2 + a_2^2Y_2^2 + S\right|^{\frac p 2} \ge \kappa_p\left(a_1^4 + a_2^4 - b_1^4 - b_2^4\right).
    \end{align*}
    The function $[0, \infty) \ni t \mapsto -t^{\frac p 2}$ is convex and hence by Lemma \ref{lem:gaussian-mixtures-auxiliary} we have
    \begin{align*}
        \E\left|a_1^2Y_1^2 + a_2^2Y_2^2 + S\right|^{\frac p 2} - \E\left|b_1^2Y_1^2 + b_2^2Y_2^2 + S\right|^{\frac p 2} \ge c_p\left(a_1^4 + a_2^4 - b_1^4 - b_2^4\right)\E\left|S + \eta\sigma^2\right|^{\frac p 2 - 2},
    \end{align*}
    where $\sigma^2 = a_1^2 + a_2^2 = b_1^2 + b_2^2$, $c_p = -\frac{cp(p-2)}{4} > 0$ and $c$, $\eta$ are positive constants depending only on the distribution of $Y_1$ (and hence only on the distribution of $X_1$). Thus we only need to get rid of the factor $\E\left|S + \eta\sigma^2\right|^{\frac p 2 - 2}$. Denote $r = \frac p 2 - 2$. Markov's inequality yields
    \begin{align*}
        \pp\left(S \le 2\E Y_1^2\right) = 1 - \pp\left(S > 2\E Y_1^2\right) \ge 1 - \frac{\E S}{2\E Y_1^2} = 1 - \frac{1 - \sigma^2}{2} \ge \frac 1 2
    \end{align*}
    and hence
    \begin{align*}
        \E\left|S + \eta\sigma^2\right|^r \ge \pp\left(S \le 2\E Y_1^2\right)\left|2\E Y_1^2 + \eta\right|^r \ge \frac 1 2 \left|2\E Y_1^2 + \eta\right|^r
    \end{align*}
    since $\sigma^2 \le 1$. As the RHS depends only on the distribution of $Y_1$, this completes the proof.

    \textit{Case $2 < p \le 4$.} Using the convexity of $t \mapsto t^{\frac p 2}$, by Lemma \ref{lem:gaussian-mixtures-auxiliary} we get
    \begin{align*}
        \E\left|a_1^2Y_1^2 + a_2^2Y_2^2 + S\right|^{\frac p 2} - \E\left|b_1^2Y_1^2 + b_2^2Y_2^2 + S\right|^{\frac p 2} \ge c_p\left(a_1^4 + a_2^4 - b_1^4 - b_2^4\right)\E\left|S + \eta\sigma^2\right|^r.
    \end{align*}
    Proceeding analogously to the previous case, using the fact that the function $t \mapsto t^r$ is non-increasing, we complete the proof.

    \medskip

    \textit{Case $p > 4$.} As in the previous cases, by Lemma \ref{lem:gaussian-mixtures-auxiliary} we get
    \begin{align*}
        \E\left|a_1^2Y_1^2 + a_2^2Y_2^2 + S\right|^{\frac p 2} - \E\left|b_1^2Y_1^2 + b_2^2Y_2^2 + S\right|^{\frac p 2} \ge c_p\left(a_1^4 + a_2^4 - b_1^4 - b_2^4\right)\E\left|S + \eta\sigma^2\right|^r
    \end{align*}
    and we need to show that
    \begin{align*}
        \E\left|a_1^2Y_1^2 + a_2^2Y_2^2 + S\right|^{\frac p 2} - \E\left|b_1^2Y_1^2 + b_2^2Y_2^2 + S\right|^{\frac p 2} \ge \kappa_p\left(a_1^4 + a_2^4 - b_1^4 - b_2^4\right).
    \end{align*}
    Thus it suffices to bound $\E\left|S + \eta\sigma^2\right|^r$ from below by a constant depending only on $p$ and distribution of $Y_1$. We have
    \begin{align*}
        \E S^2 &= \sum_{i=3}^na_i^4\E Y_i^4 + 2\sum_{3 \le i < j \le n}a_i^2a_j^2\E Y_i^2 \E Y_j^2 \le \left(\sum_{i=3}^na_i^2\right)^2\max\left\{\E Y_1^4, \left(\E Y_1^2\right)^2\right\} \\ &= \left(1 - \sigma^2\right)^2\max\left\{\E Y_1^4, \left(\E Y_1^2\right)^2\right\}.
    \end{align*}
    Note that $\E Y_1^4 < \infty$ since $\E|X_1|^p < \infty$ and $p > 4$. Now observe that this case differs from the previous ones in that the function $t \mapsto t^r$ is non-decreasing. We therefore aim to apply the Paley–Zygmund inequality rather than Markov’s inequality. Note that for non-decreasing $\varphi''$ in Lemma \ref{lem:gaussian-mixtures-auxiliary} we may guarantee that $\eta < 1$. Hence, for $\sigma^2 < 1$, by the Paley-Zygmund inequality we have
    \begin{align*}
        \pp\left(S \ge \eta\left(1 - \sigma^2\right)\E Y_1^2\right) \ge (1 - \eta)^2\frac{\left(1 - \sigma^2\right)^2\left(\E Y_1^2\right)^2}{\E S^2} \ge \frac{(1 - \eta)^2\left(\E Y_1\right)^2}{\max\left\{\E Y_1^4, \left(\E Y_1^2\right)^2\right\}},
    \end{align*}
    which gives
    \begin{align*}
        \E\left|S + \eta\sigma^2\right|^r \ge \E\1_{S \ge \eta\left(1 - \sigma^2\right)}\left|S + \eta\sigma^2\right|^r \ge \frac{(1 - \eta)^2\eta^r\left(\E Y_1^2\right)^2}{\max\left\{\E Y_1^4, \left(\E Y_1^2\right)^2\right\}}.
    \end{align*}
    If $\sigma^2 = 1$ then simply $S = 0$ almost surely and $\E\left|S + \eta\sigma^2\right|^r = \eta^r$. This completes the proof.
\end{proof}

\section{Proof of Theorem \ref{thm:lp-projection-stability}}

\subsection{The case $q = \infty$ (the cube)}
Let $I$ be uniformly distributed on $\{1,\ldots,n\}$ and take the random variable $X=\sqrt{n}|a_I|$. Then, $$\mathbb{P}(X=x)=\frac{\#\{i:\sqrt{n}|a_i|=x \}}{n}.$$
Since $\sum_{i=1}^na_i^2=1$, we have
\[
\E X^2=1
\qquad\text{and}\qquad
m:=\E X=\frac1{\sqrt n}\sum_{i=1}^n|a_i|.
\]
Consequently,
\begin{equation}\label{eq:cube-variance}
\operatorname{Var}(X)
=1-m^2
=\frac1{2n}\sum_{i,j=1}^n
\bigl(a_i-a_j\bigr)^2.
\end{equation}
Notice also that
\begin{equation}\label{eq:cube-variance-mean}
1-m
=
\frac{\operatorname{Var}(X)}{1+m}
\geq
\frac12\operatorname{Var}(X).
\end{equation}
By \eqref{eq:lp-projection-formulas},
\[
\frac{h_{\Pi_pB_{\infty}^n}(a)}{h_{\Pi_pB_{\infty}^n}\left(a_{(n)}\right)}
=
\left(n^{p/2-1}\sum_{i=1}^n|a_i|^p\right)^{1/p}
=\|X\|_p.
\]
On the other hand, one can easily see that 
\begin{equation}\label{h/h(a_1)}
\frac{h_{\Pi_pB_{\infty}^n}(a)}{h_{\Pi_pB_{\infty}^n}\left(a_{(1)}\right)}
=\|a\|_p.
\end{equation}

We first prove the cube estimates in the case $1\leq p<2$.
From~\eqref{h/h(a_1)} it is enough to estimate $\|a\|_p-1$. We have
$\|a-a_{(1)}\|_2^2=2(1-a_1)$. Suppose first that
$a_1\geq1/\sqrt2$, and set $u=1-a_1^2$. Since
\[
\sum_{i=2}^n a_i^p
\geq
\left(\sum_{i=2}^n a_i^2\right)^{p/2}
=u^{p/2},
\]
we obtain
\begin{align*}
\|a\|_p^p
&\geq
(1-u)^{p/2}+u^{p/2}\\
&\geq
1-u+u^{p/2}\\
&=
1+u^{p/2}\bigl(1-u^{1-p/2}\bigr)\\
&\geq
1+\bigl(1-2^{p/2-1}\bigr)u^{p/2}.
\end{align*}
Here we used $0\leq u\leq1/2$. Since
$0\leq(1-2^{p/2-1})u^{p/2}\leq1$, the elementary estimate
\[
(1+s)^{1/p}-1\geq\frac{s}{2p},
\qquad 0\leq s\leq1,
\]
gives
\begin{align*}
\|a\|_p-1
&\geq
\frac{1-2^{p/2-1}}{2p}u^{p/2}\\
&\geq
\frac{1-2^{p/2-1}}{2p\,2^{p/2}}
\|a-a_{(1)}\|_2^p,
\end{align*}
where we also used
$u=(1-a_1)(1+a_1)\geq1-a_1
=\frac12\|a-a_{(1)}\|_2^2$.

If $a_1<1/\sqrt2$, then, summing only over the nonzero coordinates,
\[
\|a\|_p^p
=
\sum_{i=1}^na_i^2a_i^{p-2}
\geq
a_1^{p-2}\sum_{i=1}^na_i^2
=a_1^{p-2}.
\]
Thus, since $p-2<0$,
\[
\|a\|_p-1
\geq
2^{\frac{2-p}{2p}}-1.
\]
On the other hand,
$\|a-a_{(1)}\|_2^p\leq2^{p/2}$, and hence
\[
\|a\|_p-1
\geq
2^{-p/2}\left(2^{\frac{2-p}{2p}}-1\right)
\|a-a_{(1)}\|_2^p.
\]
Therefore, \eqref{eq:cube-coordinate-stability-below-two} holds, for
example, with
\[
c_{p,\infty}
=
2^{-p/2}\min\left\{
\frac{1-2^{p/2-1}}{2p},
2^{\frac{2-p}{2p}}-1
\right\}.
\]

Still assuming that $1\leq p<2$, the interpolation inequality gives
\[
\|X\|_p
\leq
\|X\|_1^{\frac{2-p}{p}}
\|X\|_2^{\frac{2(p-1)}p}
=m^{\frac{2-p}{p}}.
\]
Since $0<m\leq1$ and $0<(2-p)/p\leq1$, the concavity of
$s\mapsto s^{(2-p)/p}$ gives
\[
1-\|X\|_p
\geq
1-m^{\frac{2-p}{p}}
\geq
\frac{2-p}{p}(1-m)
\geq
\frac{2-p}{2p}\operatorname{Var}(X).
\]
Together with \eqref{eq:cube-variance}, this proves
\eqref{eq:cube-diagonal-stability-below-two}.

Let now $p>2$. We interpolate $L_2$ between $L_1$ and $L_p$:
\[
\|X\|_2
\leq
\|X\|_1^\theta\|X\|_p^{1-\theta},
\qquad
\theta=\frac{p-2}{2(p-1)}.
\]
Since $\|X\|_2=1$, it follows that
\[
\|X\|_p
\geq
m^{-\frac{p-2}{p}}.
\]
The convexity of $s\mapsto s^{-(p-2)/p}$ therefore gives
\[
\|X\|_p-1
\geq
m^{-\frac{p-2}{p}}-1
\geq
\frac{p-2}{p}(1-m)
\geq
\frac{p-2}{2p}\operatorname{Var}(X).
\]
Using \eqref{eq:cube-variance} again proves
\eqref{eq:cube-diagonal-stability-above-two}.

For the other bound, since $p-2>0$, we have
\[
\|a\|_p^p
=
\sum_{i=1}^na_i^2a_i^{p-2}
\leq
a_1^{p-2}\sum_{i=1}^na_i^2
=a_1^{p-2}.
\]
Consequently,
\[
\|a\|_p
\leq
a_1^{\frac{p-2}{p}}
\leq
1-\frac{p-2}{p}(1-a_1)
=
1-\frac{p-2}{2p}\|a-a_{(1)}\|_2^2,
\]
where the second inequality follows from the concavity of
$s\mapsto s^{(p-2)/p}$. This proves
\eqref{eq:cube-coordinate-stability-above-two}.

\subsection{The case $q = 1$ (the cross-polytope)} 
Following the observation made in \cite{jakimiuk2026stability}, we base our proof on the idea of applying \cite[Theorem~1.3]{jendrej2015extensions}.
Set
\[
S=\sum_{i=1}^na_i\e_i.
\]
Then $\E S^2=1$. By
\cite[Theorem~1.3]{jendrej2015extensions}, there exists an index $k$
such that
\[
\sum_{i\neq k}a_i^2
\leq
C_0\operatorname{Var}(|S|).
\]
Since $a_1=\max_i|a_i|$, we obtain
\begin{equation}\label{eq:fkn-coordinate-distance}
1-a_1^2
\leq
1-a_k^2
\leq
C_0\operatorname{Var}(|S|).
\end{equation}
Moreover,
\begin{equation}\label{eq:cross-coordinate-distance}
\|a-a_{(1)}\|_2^2
=2(1-a_1)
\leq
2(1-a_1^2).
\end{equation}

Let $m=\E|S|$. Since $\E|S|^2=1$, we have $0<m\leq1$ and
\[
1-m
=
\frac{\operatorname{Var}(|S|)}{1+m}
\geq
\frac12\operatorname{Var}(|S|).
\]
If $1\leq p<2$, interpolation gives
\[
\|S\|_p
\leq
m^{\frac{2-p}{p}},
\]
and hence
\begin{align*}
1-\|S\|_p
&\geq
\frac{2-p}{p}(1-m)\\
&\geq
\frac{2-p}{2p}\operatorname{Var}(|S|)\\
&\geq
\frac{2-p}{2pC_0}(1-a_1^2)\\
&\geq
\frac{2-p}{4pC_0}\|a-a_{(1)}\|_2^2.
\end{align*}
Together with \eqref{eq:lp-cross-polytope-formula}, this proves
\eqref{eq:cross-coordinate-stability-below-two}.

If $p>2$, interpolation gives instead
\[
\|S\|_p
\geq
m^{-\frac{p-2}{p}}.
\]
It follows that
\begin{align*}
\|S\|_p-1
&\geq
\frac{p-2}{p}(1-m)\\
&\geq
\frac{p-2}{2p}\operatorname{Var}(|S|)\\
&\geq
\frac{p-2}{2pC_0}(1-a_1^2)\\
&\geq
\frac{p-2}{4pC_0}\|a-a_{(1)}\|_2^2.
\end{align*}
This proves \eqref{eq:cross-coordinate-stability-above-two}.

\subsection{The case $2 < q < \infty$}

Let $X_1, \ldots, X_n$ be the same as in Proposition \ref{prop:Lp-projection-formula} and $G$ be a Gaussian random variable with mean 0 and variance $\E|X_1|^2$. By the Schur-monotonicity of moments of Gaussian mixtures \cite[Theorem 3]{eskenazis2018gaussian}, the central limit theorem and \eqref{eq:Lp-projection-probabilistic-formula} we have
\begin{align*}
    \min\left\{\E|X_1|^p, \E|G|^p\right\} \le \E\left|\sum_{i=1}^na_iX_i\right|^p \le \max\left\{\E|X_1|^p, \E|G|^p\right\}
\end{align*}
for any $a = (a_1, \ldots, a_n) \in \Sp^{n-1}$. This is a two-sided bound by a constant depending only on $p$ and $q$. Hence, again by \eqref{eq:Lp-projection-probabilistic-formula}, the inequalities \eqref{eq:q-above-two-coordinate-stability-below-two}, \eqref{eq:q-above-two-diagonal-stability-below-two}, \eqref{eq:q-above-two-coordinate-stability-above-two}, and \eqref{eq:q-above-two-diagonal-stability-above-two} reduce to the inequalities
\begin{align}\label{eq:q-above-two-coordinate-stability-below-two-reduced}
    \E\left|\sum_{i=1}^na_iX_i\right|^p - \E|X_1|^p \ge c_{p, q}\left|a - a_{(1)}\right|^2 \quad \text{for} \quad 1 \le p < 2,
\end{align}
\begin{align}\label{eq:q-above-two-diagonal-stability-below-two-reduced}
    \E\left|\frac{1}{\sqrt{n}}\sum_{i=1}^nX_i\right|^p - \E\left|\sum_{i=1}^na_iX_i\right|^p \ge c_{p, q}\sum_{i=1}^n\left(a_i^2 - \frac 1 n\right)^2 \quad \text{for} \quad 1 \le p < 2,
\end{align}
\begin{align}\label{eq:q-above-two-coordinate-stability-above-two-reduced}
    \E|X_1|^p - \E\left|\sum_{i=1}^na_iX_i\right|^p \ge c_{p, q}\left|a - a_{(1)}\right|^2 \quad \text{for} \quad 2 < p < \infty
\end{align}
and
\begin{align}\label{eq:q-above-two-diagonal-stability-above-two-reduced}
    \E\left|\sum_{i=1}^na_iX_i\right|^p - \E\left|\frac{1}{\sqrt{n}}\sum_{i=1}^nX_i\right|^p \ge c_{p, q}\sum_{i=1}^n\left(a_i^2 - \frac 1 n\right)^2 \quad \text{for} \quad 2 < p < \infty,
\end{align}
respectively.            
        
The inequalities \eqref{eq:q-above-two-diagonal-stability-below-two-reduced} and \eqref{eq:q-above-two-diagonal-stability-above-two-reduced} follow directly from Theorem \ref{thm:gaussian-mixtures} with $b_i = \frac{1}{\sqrt{n}}$ since
\begin{align*}
    \sum_{i=1}^n\left(a_i^2 - \frac 1 n\right)^2 = \sum_{i=1}^na_i^4 - \frac 1 n = \sum_{i=1}^na_i^4 - \sum_{i=1}^n\left(\frac{1}{\sqrt{n}}\right)^4.
\end{align*}
For the inequalities \eqref{eq:q-above-two-coordinate-stability-below-two-reduced} and \eqref{eq:q-above-two-coordinate-stability-above-two-reduced} first recall that
\begin{align*}
    \left|a - a_{(1)}\right|^2 = 2 - 2a_1.
\end{align*}
Now we fix $1 \le p < 2$. We have by Theorem \ref{thm:gaussian-mixtures} and the assumption that $a_1 = \max_{1 \le i \le n}a_i$ that
\begin{align*}
    \E\left|\sum_{i=1}^na_iX_i\right|^p - \E|X_1|^p \ge C_{p, q}\left(1 - \sum_{i=1}^na_i^4\right) \ge  C_{p, q}\left(1 - a_1^2\sum_{i=1}^na_i^2\right) \ge C_{p, q}(1 - a_1),
\end{align*}
where the last estimate follows from $\sum_{i=1}^na_i^2 = 1$ and $1 - a_1^2 \ge 1 - a_1$. It resolves the inequality \eqref{eq:q-above-two-coordinate-stability-below-two-reduced}. The inequality \eqref{eq:q-above-two-coordinate-stability-above-two-reduced} follows analogously, which completes the proof.

\subsection{The case $1 < q < 2$}

For $p = 1$ this case was proved in \cite{jakimiuk2026stability}. For $p > 1$ we shall follow the same strategy. However, now this strategy involves dealing with the function $\phi_z(t) = \left|z + t^{\frac 1 p}\right|^p + \left|z - t^{\frac 1 p}\right|^p$ defined for $z \in \R$ and $t > 0$, which for $p > 1$ does not admit the representation $\frac 1 2 |t + z| + \frac 1 2 |t - z| = \max\{t, |z|\} = t + (|z| - t)_+$, crucial in the proof of \cite[Theorem 1(d)]{jakimiuk2026stability}. In order to find a similar representation for $p > 1$, we shall first prove the following integral formula for $\phi_z''$, which will be the basis of further quantitative improvements of convexity arguments. For simplicity we shall consider the function $\phi_z$ multiplied by $2 - p$.

\begin{lemma}\label{lem:integral-formula-auxiliary-function}
    Let $p > 1$ and $\phi_z(t) = (2-p)\left(\left|z + t^{\frac 1 p}\right|^p + \left|z - t^{\frac 1 p}\right|^p\right)$ for $z \in \R$ and $t > 0$. Then for $t \ne z^p$ we have
    \begin{align*}
        \phi_z''(t) = \frac{(2-p)^2(p-1)}{p}zt^{-\frac{p+1}{p}}\int_{\left|zt^{-\frac 1 p} - 1\right|}^{zt^{-\frac 1 p} + 1}s^{p-3}ds
    \end{align*}
\end{lemma}

\begin{proof}
    Without loss of generality we assume that $z \ge 0$. By standard computation we get
    \begin{align*}
        \phi_z'(t) &= (2-p)t^{\frac 1 p - 1}\left(\left|z + t^{\frac 1 p}\right|^{p-1} + \left|z - t^{\frac 1 p}\right|^{p-1}\sgn\left(t^{\frac 1 p} - z\right)\right) \\ &= (2-p)\left(\left|zt^{-\frac 1 p} + 1\right|^{p-1} + \left|zt^{-\frac 1 p} - 1\right|^{p-1}\sgn\left(1 - zt^{-\frac 1 p}\right)\right)
    \end{align*}
    for all $t > 0$. Note that differentiability at $t = z^p$ is guaranteed by $p > 1$. For $t \ne z^p$ we can differentiate again to obtain
    \begin{align*}
        \phi_z''(t) &= \frac{(2-p)(p-1)}{p}zt^{-\frac 1 p - 1}\left(\left|zt^{-\frac 1 p} - 1\right|^{p-2} + \left|zt^{-\frac 1 p} + 1\right|^{p-2}\right) \\ &= \frac{(2-p)^2(p-1)}{p}zt^{-\frac{p+1}{p}}\int_{\left|zt^{-\frac 1 p} - 1\right|}^{zt^{-\frac 1 p} + 1}s^{p-3}ds.
    \end{align*}
\end{proof}

\begin{corollary}
    It follows immediately that the function $\phi_z$ is convex for any $p > 1$. This is also the content of \cite[Lemma~12(i)]{eskenazis2018sharp}.
\end{corollary}

The approach based on \cite{jakimiuk2026stability} relies heavily on comparing the random variables $|X_i|^p$ from Proposition \ref{prop:Lp-projection-formula} with $|G|^p$ in the convex order, where $G$ is a Gaussian random variable. For $p = 1$ the necessary comparison was proved in \cite{barthe2002hyperplane}, but for $p > 1$ we need to develop it. Below we prove a more general lemma (note that the random variables $X_i$ are Gaussian for $q = 2$), which shall be useful also in the proof of Theorem \ref{thm:monotonicity-normalized-lp-projections}.

\begin{lemma}\label{lem:projection-variables-convex-order}
Fix $p>0$. For $q>1$, let $X_q$ be a random variable with
density proportional to
\[
|x|^{\frac{2-q}{q-1}}
\exp\left(-|x|^{\frac{q}{q-1}}\right),
\qquad x\in\mathbb R,
\]
and set
\[
T_q:=\frac{|X_q|^p}{\mathbb E|X_q|^p}.
\]
Then, for every $1<q_1<q_2<\infty$,
\begin{equation}\label{eq:projection-variables-convex-order}
T_{q_1}\leq_{\mathrm{cx}}T_{q_2},
\end{equation}
where $\leq_{\mathrm{cx}}$ denotes convex order. That is,
\[
\mathbb E\varphi(T_{q_1})
\leq
\mathbb E\varphi(T_{q_2})
\]
for every convex function $\varphi:[0,\infty)\to\mathbb R$
for which the expectations are finite.
\end{lemma}

\begin{proof}
An easy computation shows that the density of $T_q$ is
\[
f_q(t)=c_qt^{\alpha_q}\exp(-d_qt^{\beta_q}),
\qquad t>0,
\]
where $c_q,d_q>0$ and
\[
\alpha_q=\frac{1}{p(q-1)}-1,
\qquad
\beta_q=\frac{q}{p(q-1)}.
\]
Both $\alpha_q$ and $\beta_q$ are strictly decreasing in $q$.
Therefore, \cite[Lemma~9]{barthe2002hyperplane} implies that there
exist $0<A<B<\infty$ such that $f_{q_2}-f_{q_1}$ is
nonnegative outside $(A,B)$ and negative on $(A,B)$.
The difference cannot be nonnegative everywhere, since both densities have total mass one and are
distinct.

Moreover, by the choice of normalization,
\[
\int_0^\infty f_{q_j}(t)\,dt=1,
\qquad
\int_0^\infty t f_{q_j}(t)\,dt=1,
\qquad j=1,2.
\]
We may thus apply \cite[Lemma~8]{barthe2002hyperplane} to the
symmetric probability densities
\[
\widetilde f_{q_j}(x):=\frac12 f_{q_j}(|x|),
\qquad x\neq 0.
\]
Testing the resulting symmetric Choquet ordering against the
even convex functions $x\mapsto (|x|-s)_+$ gives
\[
\mathbb E(T_{q_1}-s)_+
\leq
\mathbb E(T_{q_2}-s)_+,
\qquad s\geq 0.
\]
Indeed,
\[
\int_{\R}(|x|-s)_+\widetilde{f}_q(x)=\frac{1}{2}\int_{\R}(|x|-s)_+ f_q(|x|)=\int_{0}^{\infty}(x-s)_+f_q(x)\, dx.
\]
 Since the variables $T_q$ are not symmetric, we use the nonsymmetric
counterpart of \cite[Lemma~6]{barthe2002hyperplane}, namely the standard
characterization of convex order (see for instance
\cite[Theorem~3.A.1(a)]{shaked2007stochastic}) yielding the result.
\end{proof}

Now we are ready to prove estimates \eqref{eq:q-below-two-coordinate-stability-below-two} and \eqref{eq:q-below-two-coordinate-stability-above-two}.

\begin{proof}[Proof of estimates \eqref{eq:q-below-two-coordinate-stability-below-two} and \eqref{eq:q-below-two-coordinate-stability-above-two}]
    The case $p = 1$ is already proved, so we assume that $p > 1$. Let $X_1, \ldots, X_n$ be the same as in Proposition \ref{prop:Lp-projection-formula}, normalized such that $\E|X_1|^p = 1$. By \eqref{eq:Lp-projection-probabilistic-formula}, the inequalities \ref{eq:q-below-two-coordinate-stability-below-two} and \ref{eq:q-below-two-coordinate-stability-above-two} reduce to the inequalities
    \begin{align}\label{eq:q-below-two-coordinate-stability-below-two-reduced}
        \E\left|\sum_{i=1}^na_iX_i\right|^p \le 1 - c_{p, q}(1 - a_1)^p \quad \text{for} \quad 1 < p < 2
    \end{align}
    and
    \begin{align}\label{eq:q-below-two-coordinate-stability-above-two-reduced}
        \E\left|\sum_{i=1}^na_iX_i\right|^p \ge 1 + c_{p, q}(1 - a_1)^p \quad \text{for} \quad p > 2,
    \end{align}
    respectively. Formally we should rise the LHS to the power $\frac 1 p$, but clearly it affects only the dependence on $p$ of $c_{p, q}$, therefore we ignore it. We split the proof into two cases.

    \noindent \textit{Case 1.} $a_1 \ge c_{p, q}'$, where $c_{p, q}' > 0$ is a constant to be chosen later. Let $G_1, \ldots, G_n$ be i.i.d. Gaussian random variables independent of variables $X_1, \ldots, X_n$, normalized such that $\E|G_1|^p = 1$. To avoid further splitting into cases we shall write \eqref{eq:q-below-two-coordinate-stability-below-two-reduced} and \eqref{eq:q-below-two-coordinate-stability-above-two-reduced} jointly as
    \begin{align}\label{eq:q-below-two-coordinate-stability-reduced}
        (2 - p)\left(\E\left|\sum_{i=1}a_iG_i\right|^p - \E\left|\sum_{i=1}^na_iX_i\right|^p\right) \ge |2 - p|c_{p, q}(1 - a_1)^p.
    \end{align}
    Let $Z = \sum_{i=2}^na_iG_i$, then $Z \sim \sqrt{1 - a_1^2}G_1$. Using the Lindeberg swapping argument, we see that
    \begin{align*}
        \E|a_1X_1 + Z|^p - \E\left|\sum_{i=1}^na_iX_i\right|^p = \sum_{i=2}^n\left(\E\left|\sum_{j < i}a_jX_j + \sum_{j \ge i}a_jG_j\right|^p - \E\left|\sum_{j \le i}a_jX_j + \sum_{j > i}a_jG_j\right|^p\right).
    \end{align*}
    Consider the $i$-th term in the sum on the RHS. It follows from Lemma \ref{lem:projection-variables-convex-order} with $q_1 = q$, $q_2 = 2$, from the convexity of $\phi_z$ and from symmetry of the random variables under consideration that for any $z \in \R$ we have
    \begin{align*}
        (2-p)\E|a_iG_i + z|^p = \frac 1 2\E\phi_z\left(a_i^p|G_i|^p\right) \ge \frac 1 2\E\phi_z\left(a_i^p|X_i|^p\right) = (2-p)\E|a_iX_i + z|^p,
    \end{align*}
    which together with conditioning on all variables with indices $j \ne i$ shows that 
    \begin{align*}
        (2-p)\sum_{i=2}^n\left(\E\left|\sum_{j < i}a_jX_j + \sum_{j \ge i}a_jG_j\right|^p - \E\left|\sum_{j \le i}a_jX_j + \sum_{j > i}a_jG_j\right|^p\right) \ge 0.
    \end{align*}
    Thus to prove \eqref{eq:q-below-two-coordinate-stability-reduced} it suffices to show that
    \begin{align}\label{eq:q-below-two-coordinate-stability-reduced-further}
        (2-p)\left(\E|a_1G_1 + Z|^p - \E|a_1X_1 + Z|^p\right) \ge |2-p|c_{p, q}(1 - a_1)^p.
    \end{align}
    Using symmetry of the random variables under consideration as before we get
    \begin{align}\label{eq:q-below-two-coordinate-stability-convex-order}
        (2-p)\left(\E|a_1G_1 + Z|^p - \E|a_1X_1 + Z|^p\right) = \frac 1 2\E\left[\phi_Z\left(a_1^p|G_1|^p\right) - \phi_Z\left(a_1^p|X_1|^p\right)\right].
    \end{align}
    
    Consider the functions $f_q(x) = c_qx^{\alpha_q}\exp(-d_qx^{\beta_q})$, $f_2(x) = c_2x^{\alpha_2}\exp(-d_2x^{\beta_2})$ from the proof of Lemma \ref{lem:projection-variables-convex-order}, being the densities of $|X_1|^p$ and $|G_1|^p$, respectively. Since $\alpha_2 < \alpha_q$ and $\alpha_2 = \frac 1 p - 1 < 0$, there exist numbers $b, \delta > 0$ such that $f_2(x) - f_q(x) \ge b > 0$ for $0 < x \le \delta$. Fix $0 \le z < a_1\delta^{\frac 1 p}$ for a moment. Clearly we have
    \begin{align*}
        \E\left[\phi_z\left(a_1^p|G_1|^p\right) - \phi_z\left(a_1^p|X_1|^p\right)\right] = \int_0^{\infty}\phi_z\left(a_1^px\right)(f_2(x) - f_q(x))dx.
    \end{align*}
    Now we need a representation of $\phi_z$ allowing to develop quantitative improvements of convexity of $\phi_z$ and take advantage on the observation that $\int_0^{\infty}\varphi(x)(f_2(x) - f_q(x))dx = 0$ for any affine function $\varphi$, which follows from $\E|X_1|^p = \E|G|^p$. Taking the first-order Taylor expansion with an explicit remainder seems to be a natural approach, since it provides a decomposition into an affine part and a convex part involving the second derivative of $\phi_z$, for which we have a nice formula from Lemma \ref{lem:integral-formula-auxiliary-function}. It also seems natural to take this expansion around the point $z^p$. Assumptions of the Taylor's theorem are clearly satisfied since $\phi_z$ is continuously differentiable in $(0, \infty)$ and twice differentiable in $(0, \infty) \backslash \{z^p\}$.

    Thus, using the Taylor expansion with integral remainder and the convention that $\int_u^v = -\int_v^u$ for $u > v$, we have
    \begin{align*}
        \phi_z\left(a_1^px\right) = \phi_z\left(z^p\right) + \phi_z'\left(z^p\right)\left(a_1^px - z^p\right) + \int_{z^p}^{a_1^px}\phi_z''(t)\left(a_1^px - t\right)dt
    \end{align*}
    and hence
    \begin{align*}
        \int_0^{\infty}\phi_z\left(a_1^px\right)(f_2(x) - f_q(x))dx = \int_0^{\infty}(f_2(x) - f_q(x))\int_{z^p}^{a_1^px}\phi_z''(t)\left(a_1^px - t\right)dtdx,
    \end{align*}
    which by Lemma \ref{lem:integral-formula-auxiliary-function} is equal to
    \begin{align*}
        \frac{(2-p)^2(p-1)}{p}z\int_0^{\infty}(f_2(x) - f_q(x))\int_{z^p}^{a_1^px}\left(a_1^px - t\right)t^{-\frac{p+1}{p}}\int_{\left|zt^{-\frac 1 p} - 1\right|}^{zt^{-\frac 1 p} + 1}s^{p-3}dsdtdx = I.
    \end{align*}
    Now we shall decompose $I = \int_0^{\infty}dx = \int_0^{z^pa_1^{-p}}dx + \int_{z^pa_1^{-p}}^{\infty}dx = I_1 + I_2$. The choice of $z$ guarantees that all factors in the integral $I_1$ are positive and allows us to derive a quantitative estimate from below by crude bounds. Then we only need to ensure that $I_2 \ge 0$. There the factor $f_2(x) - f_q(x)$ is the only one which could be negative. We are going to overcome this difficulty by using the Fubini's theorem to put the integral with respect to $x$ inside the other integrals and thus make this integral easier to control.

    Denote $\tilde{c}_p = \frac{(2-p)^2(p-1)}{p}$. Clearly $f_2(x) - f_q(x) \ge b$ for $x \le z^pa_1^{-p}$. Note that for every $u > 1$ and $r \in \R$ the inequality $\int_{u-1}^{u+1}s^rds \ge u^r$ holds. Indeed, for $r = 0$ it is trivial, for $r > 0$ we have $\int_{u-1}^{u+1}s^rds \ge \int_u^{u+1}u^rds = u^r$ by monotonicity and for $r < 0$ we have similarly $\int_{u-1}^{u+1}s^rds \ge \int_{u-1}^uu^rds = u^r$. Thus we get
    \begin{align}\label{eq:integral-one-bound}
        I_1 &= z\tilde{c}_p\int_0^{z^pa_1^{-p}}(f_2(x) - f_q(x))\int_{a_1^px}^{z^p}\left(t - a_1^px\right)t^{-\frac{p+1}{p}}\int_{zt^{-\frac 1 p} - 1}^{zt^{-\frac 1 p} + 1}s^{p-3}dsdtdx \notag \\ &\ge z^{p-2}\tilde{c}_p\int_0^{z^pa_1^{-p}}(f_2(x) - f_q(x))\int_{a_1^px}^{z^p}\left(t - a_1^px\right)t^{-\frac{p+1}{p} - \frac{p-3}{p}}dtdx \notag \\ &\ge z^{p-2}\tilde{c}_p\int_0^{z^pa_1^{-p}}(f_2(x) - f_q(x))\int_{a_1^px}^{z^p}\left(t - a_1^px\right)z^{-(2p-2)}dtdx \notag \\ &= 2^{-1}z^{-p}\tilde{c}_p\int_0^{z^pa_1^{-p}}(f_2(x) - f_q(x))\left(z^p - a_1^px\right)^2dx \ge 2^{-1}z^{-p}b\tilde{c}_p\int_0^{z^pa_1^{-p}}\left(z^p - a_1^px\right)^2dx \notag \\ &= 2^{-1}z^{-p}a_1^{-p}b\tilde{c}_p\int_0^{z^p}y^2dy = \frac{b\tilde{c}_pz^{2p}}{6a_1^p}.
    \end{align}
    It remains to prove that
    \begin{align}\label{eq:integral-two-bound}
        I_2 = z\tilde{c}_p\int_{z^pa_1^{-p}}^{\infty}(f_2(x) - f_q(x))\int_{z^p}^{a_1^px}\left(a_1^px - t\right)t^{-\frac{p+1}{p}}\int_{1 - zt^{-\frac 1 p}}^{1 + zt^{-\frac 1 p}}s^{p-3}dsdtdx \ge 0.
    \end{align}
    We need to ensure that the integrability assumptions of Fubini's theorem are satisfied. Since the only source of varying sign is the factor $f_2(x) - f_q(x)$, it sufficed to show that
    \begin{align*}
        z\tilde{c}_p\int_{z^pa_1^{-p}}^{\infty}|f_2(x) - f_q(x)|\int_{z^p}^{a_1^px}\left(a_1^px - t\right)t^{-\frac{p+1}{p}}\int_{1 - zt^{-\frac 1 p}}^{1 + zt^{-\frac 1 p}}s^{p-3}dsdtdx < \infty.
    \end{align*}
    It follows by writing the LHS above as
    \begin{align*}
        \int_{z^pa_1^{-p}}^{\infty}|f_2(x) - f_q(x)|\int_{z^p}^{a_1^px}\left(a_1^px - t\right)\phi_z''(t)dtdx \le a_1^p\int_{z^pa_1^{-p}}^{\infty}(f_2(x) + f_q(x))x\int_{z^p}^{a_1^px}\phi_z''(t)dtdx,
    \end{align*}
    which clearly gives a finite bound due to the polynomial growth of $\phi_z'$ and the exponential vanishing of the densities $f_2$, $f_q$. Thus by Fubini's theorem we have
    \begin{align*}
        I_2 &= z\tilde{c}_p\int_0^2s^{p-3}\int_{z^p}^{\infty}t^{-\frac{p+1}{p}}\1_{1 - zt^{-\frac 1 p} \le s \le 1 + zt^{-\frac 1 p}}\int_{ta_1^{-p}}^{\infty}(f_2(x) - f_q(x))\left(a_1^px - t\right)dxdtds \\ &= z\tilde{c}_p\int_0^2s^{p-3}\int_{z^p}^{z^p|s-1|^{-p}}t^{-\frac{p+1}{p}}\int_{ta_1^{-p}}^{\infty}(f_2(x) - f_q(x))\left(a_1^px - t\right)dxdtds. 
    \end{align*}
    Now to get $I_2 \ge 0$ it remains to observe that
    \begin{align*}
        \int_{ta_1^{-p}}^{\infty}(f_2(x) - f_q(x))\left(a_1^px - t\right)dx = \int_0^{\infty}(f_2(x) - f_q(x))\left(a_1^px - t\right)_+dx \ge 0
    \end{align*}
    by convexity of $x \mapsto \left(a_1^px - t\right)_+$ and Lemma \ref{lem:projection-variables-convex-order}. Combining \eqref{eq:integral-one-bound} and \eqref{eq:integral-two-bound}, we obtain
    \begin{align}\label{eq:integral-joint-bound}
        I \ge \frac{b\tilde{c}_pz^{2p}}{6a_1^p}.
    \end{align}
    Since $\E\phi_z\left(a_1^p|G_1|^p\right) \ge \E\phi_z\left(a_1^p|X_1|^p\right)$ for any $z \in \R$, by \eqref{eq:integral-joint-bound} and $Z \sim \sqrt{1 - a_1^2}G_1$ we get
    \begin{align*}
        \E\left[\phi_Z\left(a_1^p|G_1|^p\right) - \phi_Z\left(a_1^p|X_1|^p\right)\right] &\ge \E_Z\1_{|Z| < a_1\delta^{\frac 1 p}}\E_{G_1, X_1}\left[\phi_Z\left(a_1^p|G_1|^p\right) - \phi_Z\left(a_1^p|X_1|^p\right)\right] \\ &\ge \frac{b\tilde{c}_p}{6a_1^p}\E|Z|^{2p}\1_{|Z| < a_1\delta^{\frac 1 p}} \ge \frac{b\tilde{c}_p}{6}\E|Z|^{2p}\1_{|Z| < c'_{p, q}\delta^{\frac 1 p}} \\ &= \frac{b\tilde{c}_p\left(1 - a_1^2\right)^p}{6}\E|G_1|^{2p}\1_{|G_1| < c'_{p, q}\delta^{\frac 1 p}\left(1 - a_1^2\right)^{-\frac 1 2}} \\ &\ge \frac{b\tilde{c}_p\left(1 - a_1\right)^p}{6}\E|G_1|^{2p}\1_{|G_1| < c'_{p, q}\delta^{\frac 1 p}},
    \end{align*}
    which is the bound of the desired form.

    \noindent \textit{Case 2.} $a_1 < c'_{p, q}$. In this case it suffices to prove that
    \begin{align}\label{eq:q-below-two-coordinate-stability-separated}
        \sgn(2-p)\left(1 - \E\left|\sum_{i=1}^na_iX_i\right|^p\right) \ge c
    \end{align}
    for some constant $c > 0$ depending only on $p$ and $q$. Let $G$, $Z$ be centered Gaussian random variables such that $\E|G|^p = \E|X_1|^p = 1$ and $\E|Z|^2 = \E|X_1|^2$. If $1 < p < 2$ then by Lemma \ref{lem:projection-variables-convex-order} and strict convexity of $[0, \infty) \ni x \mapsto x^{\frac 2 p}$ we know that
    \begin{align*}
        \E|G|^2 > \E|X_1|^2 = \E|Z|^2,
    \end{align*}
    whereas for $p > 2$ we have
    \begin{align*}
        \E|G|^2 < \E|X_1|^2 = \E|Z|^2.
    \end{align*}
    For simplicity we shall now assume that $p < 2$, as for $p > 2$ the reasoning is analogous. Since $1 = \E|G|^p > \E|Z|^p$, in order to obtain \eqref{eq:q-below-two-coordinate-stability-separated} we need only to show that for every $a_1 < c'_{p, q}$ for a suitable choice of $c'_{p, q}$ the quantity
    \begin{align}\label{eq:clt-quantity}
        \left|\E\left|\sum_{i=1}^na_iX_i\right|^p - \E|Z|^p\right| 
    \end{align}
    is small. Indeed, one can expect \eqref{eq:clt-quantity} to be small due to the Central Limit Theorem, but since we do not work with a fixed sequence of sums of the form $\sum_{i=1}^na_iX_i$, we need to formalize this idea. We follow the method of the proof of \cite[Theorem 1(d)]{jakimiuk2026stability}.

    Denote $S = \sum_{i=1}^na_iX_i$. We have
    \begin{align*}
        \left|\E|S|^p - \E|Z|^p\right| &= p\left|\int_0^{\infty}t^{p-1}\pp(|S| > t)dt - \int_0^{\infty}t^{p-1}\pp(|Z| > t)dt\right| \\ &\le p\int_0^Tt^{p-1}|\pp(|S| > t) - \pp(|Z| > t)|dt + p\int_T^{\infty}t^{p-1}(\pp(|S| > t) + \pp(|Z| > t))dt
    \end{align*}
    for any $T > 0$. As observed in the proof of \cite[Theorem 1(d)]{jakimiuk2026stability}, the random variable $X_1$ is subgaussian, thus $S$ admits subgaussian tail estimates and the second integral above is bounded exponentially in $T$, say that by $C_{p, q}'e^{-T}$. By the Berry-Esseen bound the first integral is bounded by $CT^p\sum_{i=1}^na_i^3 \le CT^pa_1\sum_{i=1}^na_i^2 \le CT^pc'_{p, q}$, where $C$ is the constant from the Berry-Esseen theorem. Thus if we choose $T$ such that $C'_{p, q}e^{-T} < \frac{1 - \E|Z|^p}{4}$ and then $c'_{p, q}$ such that $CT^pc'_{p, q} < \frac{1 - \E|Z|^p}{4}$, we get
    \begin{align*}
        \E|S|^p \le \E|Z|^p + \left|\E|S|^p - \E|Z|^p\right| \le \E|Z|^p + \frac{1 - \E|Z|^p}{2} = 1 - \frac{1 - \E|Z|^p}{2}.
    \end{align*}
    This completes the proof of \eqref{eq:q-below-two-coordinate-stability-separated}.
\end{proof}

\section{Proof of Theorem \ref{thm:monotonicity-normalized-lp-projections}}

Let $1<q_1<q_2<\infty$.
For $q>1$, let $X_q$ be a random variable with density proportional to
\[
|x|^{\frac{2-q}{q-1}}
\exp\left(-|x|^{\frac{q}{q-1}}\right),
\qquad x\in\mathbb R,
\]
and put
\[
Y_q:=\frac{X_q}{(\mathbb E|X_q|^p)^{1/p}},
\qquad
T_q:=|Y_q|^p.
\]
The probabilistic representation of the $L_p$ projection body gives
\begin{equation}\label{eq:normalized-lp-projection-moments}
\frac{h^p_{\Pi_p B_q^n}(a)}{h^p_{\Pi_p B_q^n}(e_1)}
=
\mathbb E\left|\sum_{i=1}^n a_iY_{q,i}\right|^p,
\qquad a\in\mathbb R^n,
\end{equation}
where $Y_{q,1},\ldots,Y_{q,n}$ are independent copies of $Y_q$.

We now apply the replacement argument used in
\cite[Theorem~10(i)]{eskenazis2018sharp}.
Let $S$ be a symmetric random variable with a finite $p$-th moment,
independent of the variables being compared, and let $\lambda\in\mathbb R$.
By symmetry,
\[
H(t):=\mathbb E|S+\lambda t^{1/p}|^p
=
\frac12\mathbb E\left[
|S+\lambda t^{1/p}|^p+|S-\lambda t^{1/p}|^p
\right],
\qquad t\geq 0.
\]
By \cite[Lemma~12(i)]{eskenazis2018sharp} and scaling,
$H$ is convex for $1\leq p\leq 2$ and concave for $p\geq 2$.
Furthermore, symmetry and independence give
\[
\mathbb E|S+\lambda Y_q|^p=\mathbb EH(T_q).
\]
It follows from \eqref{eq:projection-variables-convex-order} that
\[
\mathbb E|S+\lambda Y_{q_1}|^p
\leq
\mathbb E|S+\lambda Y_{q_2}|^p,
\qquad 1\leq p\leq 2,
\]
with the reverse inequality for $p\geq 2$.
Replacing the independent summands one at a time therefore yields
\[
\mathbb E\left|\sum_{i=1}^n u_iY_{q_1,i}\right|^p
\leq
\mathbb E\left|\sum_{i=1}^n u_iY_{q_2,i}\right|^p,
\qquad 1\leq p\leq 2,
\]
again with the reverse inequality for $p\geq 2$.
At each replacement, the remaining sum is symmetric because its
summands are independent and symmetric.
The asserted inequalities now follow from
\eqref{eq:normalized-lp-projection-moments}.

The remaining cases $q=1,\infty$ follow from a limiting argument.
Indeed, $B_q^n\to B_1^n$ as $q\to1^+$ and
$B_q^n\to B_\infty^n$ as $q\to\infty$, in the Hausdorff metric.
Then the support functions converge uniformly, and the surface
area measures converge weakly. Formula~\eqref{L_p proj body}
completes the proof.

\section*{AI tools disclosure}

All mathematical ideas, as well as their technical implementation, were developed entirely by the authors. The authors used ChatGPT 5.6 Sol and ChatGPT 6 to improve the exposition and the language, assist with literature searches, and assist in checking some of the calculations. After using these tools, the authors reviewed the content and take full responsibility for the publication.

\end{document}